\documentclass[12pt]{iopart}
 \expandafter\let\csname equation*\endcsname\relax 
 \expandafter\let\csname endequation*\endcsname\relax 
\usepackage{amsmath,amsthm,amsfonts,amssymb,amsbsy,latexsym,graphicx}
\usepackage[colorlinks=true, linkcolor=blue]{hyperref}

\usepackage[normalem]{ulem}
\newcommand{\stkout}[1]{\ifmmode\text{\sout{\ensuremath{#1}}}\else\sout{#1}\fi} 
\usepackage[dvipsnames]{xcolor}

\usepackage[capitalise,noabbrev]{cleveref}

\def\rg{\mathrm{rg }\,  }
\def\d{\mathrm{d}}
\def\p{\partial }

\def\D{\mathrm{D}}

\def\Id{\mathrm{Id}}
\def\e{\epsilon}

\def\N{\mathbb{N}}
\def\R{\mathbb{R}}
\def\diag{\mathrm{diag}\, }

\newtheorem{theorem}{Theorem}[section]
\newtheorem{assumption}{Assumption}[section]
\newtheorem{corollary}[theorem]{Corollary}
\newtheorem{lemma}[theorem]{Lemma}
\newtheorem{prop}[theorem]{Proposition}

\theoremstyle{definition}
\newtheorem{remark}[theorem]{Remark}
\newtheorem{example}[theorem]{Example}
\newtheorem{definition}[theorem]{Definition}

\begin{document}

\title[Detection of high codimensional bifurcations in variational PDEs]{Detection of high codimensional bifurcations\\
in variational PDEs}
\author{L M Kreusser$^1$ R I McLachlan$^2$ and C Offen$^2$\footnote{corresponding author}}
\address{$^1$ Department of Applied Mathematics and Theoretical Physics (DAMTP), 
University of Cambridge,
Cambridge,
United Kingdom
}
\address{$^2$ School of Fundamental Sciences, Massey University, Palmerston North, New Zealand}
\eads{\mailto{L.M.Kreusser@damtp.cam.ac.uk},
\mailto{r.mclachlan@massey.ac.nz},
 \mailto{c.offen@massey.ac.nz}
 }

\noindent{\it Keywords\/}:
bifurcations,
numerical continuation,
augmented systems,
Bell polynomials

\pacs{02.30.Jr,
02.30.Oz,
02.30.Sa}

\submitto{\NL}

\begin{abstract}
We derive bifurcation test equations for $A$-series singularities of nonlinear functionals and, based on these equations, we propose a numerical method for detecting high codimensional bifurcations in parameter-dependent PDEs such as  parameter-dependent semilinear Poisson equations. As an example, we consider a Bratu-type problem and show how high codimensional bifurcations such as the swallowtail bifurcation can be found numerically. In particular, our original contributions are (1) the use of the Infinite-Dimensional Splitting Lemma, (2) the unified and simplified treatment of all $A$-series bifurcations, 
(3) the presentation in Banach spaces, i.e.\ our results apply both to the PDE and its (variational) discretization, 
(4)~further simplifications for parameter-dependent semilinear Poisson equations (both continuous and discrete), and
(5) the unified treatment of the continuous problem and its discretisation.
\end{abstract}

\maketitle

\section{Introduction}

Nonlinear systems of ordinary (ODEs) or  partial differential equations (PDEs) are the basis for modeling many  problems in science and technology, including biological pattern formation, viscous fluid flow phenomena, chemical reactions and crystal growth processes. These mathematical models usually depend on a number of parameters and in terms of  applications  it is crucial to understand the qualitative dependence of the solution on the model parameters.

To explore the solution set of a parameter-dependent system of ordinary or partial differential equations subject to initial and/or boundary conditions, it is helpful to find parameter values and points in the phase space at which bifurcations occur, i.e.\ at which the structure of the solution set changes. Most prominent is the fold bifurcation in which two branches of solutions merge and die as one parameter is varied (see its bifurcation diagram on the left in \cref{fig:AseriesIllus}). 
A fold bifurcation is a generic phenomenon, i.e.\ it persists under small perturbations of the parameter-dependent system of differential equations.

Bifurcations can be classified into low and high codimensional bifurcations according to their codimension where the codimension of a bifurcation is defined as the
minimal number of parameters in which that bifurcation type occurs.
Low codimensional bifurcations such as folds have been studied extensively in the literature. 
Detecting such points in the parameter-phase portrait, following the solution branches and classifying them is an important task  for getting a better understanding of the physical properties of the underlying model \cite{LEGER20171}. The detection of bifurcation points and the numerical computation of the associated  bifurcation diagram is typically based on following solution branches using numerical continuation methods. Numerical continuation has been an active research area, see  \cite{abbott1978,bathe1983,crisfield1981fast,KrauskopfNumericalContinuation,rheinboldt1981numerical,riks1972application,wagner1988simple,wriggers1990general}, for instance, and dynamical systems software packages like AUTO \cite{Doedel,Doedel:2003}, CONTENT \cite{Govaerts:1998} and MATCONT \cite{Dhooge2003,Dhooge2003:proc} are available for the detection of bifurcations in ODEs.  Applying continuation methods to PDEs, however, leads to the additional challenge of large systems of equations since these methods are typically based on finite element or finite difference discretizations of PDEs, as in the MATLAB package pde2path \cite{Uecker_2014}, for instance. For the continuation process, predictor-corrector algorithms are typically used, i.e.\ nonlinear systems have to be solved in each continuation step using Newton-type methods \cite{Fedoseyev2000}. Some convergence results are available \cite{Boehmer_2010,Kunkel,paez2012}.

If more than one parameter is present in the definition of the underlying mathematical model, more complicated bifurcations may occur. In particular,  the presence of many parameters is required for bifurcation points or singular points of high codimension to occur such that they are persistent under perturbations of the dynamical system. High codimensional bifurcations act like organising centres in bifurcation diagrams, i.e.\ they determine which bifurcations happen close to  singular points \cite{Broer2007OrganisingCI}, and occur in many biological and chemical systems \cite{KERTESZ,Liu2016,Peng2008}. Therefore, it is of interest to find high codimensional singularities. Just as in detecting low codimensional bifurcations, one typically follows a branch of solutions using numerical continuation methods  \cite{KrauskopfNumericalContinuation}.
During this process it is crucial to detect singularities along the way in order to recognize other merging branches of solutions which would otherwise be missed. 

It is often useful to classify bifurcations into two principal classes: changes in the topology of the phase portrait \cite{Kuznetsov2004} and changes of the solution set.  
For example, in the recent paper \cite{Diouf2019}, the authors consider a reaction-diffusion model 
and show that by varying  three model parameters simultaneously, they could delimit several bifurcation surfaces with high codimension such as transcritical, Bogdanov-Takens, Hopf and saddle-node bifurcations. Motivated by biological models, a classification of  Bogdanov-Takens singularities is presented, and some properties of two singularities of high codimension is proven \cite{KERTESZ}. For specific examples of  neural network systems, high codimensional bifurcations are studied in \cite{Liu2016,Peng2008}. 


This paper deals with bifurcations of solutions of differential equations and in particular with those whose solutions are critical points of a functional. These bifurcations can be classified according to catastrophe theory \cite{Arnold1992} and the first three bifurcation classes of this classification are the so-called $A$-series, $D$-series and $E$-series bifurcations. 
In fact, all occurring bifurcation phenomena can be classified and related to catastrophe theory \cite{bifurHampaper,numericalPaper} for special instances, e.g.\ for systems of differential equations with Hamiltonian structure and boundary condition of Lagrangian type such as many physically relevant boundary conditions like periodic, Dirichlet or Neumann boundary conditions for second order ODEs.
In comparison to ODEs, the detection of bifurcations of solutions of PDEs poses additional challenges including infinite dimensionality and (after discretisation) the need to solve large systems of equations.
As a prototype of a PDE problem, we consider the semilinear Poisson equation 
\begin{equation*}
\left\{\begin{aligned}
\Delta u + f(u,\lambda) &= 0\\
u|_{\p \Omega} &=0
\end{aligned}\right.
\end{equation*}
for {$u$ defined on $\Omega \subset \R^d,\lambda\in \R^k$}. Its parameter-dependent solutions can be regarded as stationary solutions of an associated reaction-diffusion equation with many applications in the physical sciences.

{It seems that   maps of singularities in the Banach space setting were first considered by Ambrosetti and Prodi \cite{Ambrosetti1972} where they studied the situation of the fold map. Since then, fold and cusp maps have been studied extensively and many characterizations of fold and cusp maps have been given  \cite{Berger1985,Berger1988,Lazzeri1987,Church1992fold,Church1993cusp}.}

{According to the Ambrosetti-Prodi theorem \cite{Ambrosetti1972}, the map $F(u) = \Delta u + f(u,\lambda)$ between appropriate functional spaces is a global fold map, provided certain hypotheses such as the convexity of the function $f$ with respect to its first component are satisfied. In \cite{Calanchi}, the authors show under mild conditions that convexity is indeed necessary. If the convexity condition of $f$ is not fulfilled, there exists a point with at least four preimages under $F$ and $F$ generically admits cusps among its critical points. }

{Following the fold and the cusp, the two subsequent singularities are the swallowtail and the butterfly whose maps were characterized by Ruf  \cite{ruf1995HigherSingularities}. As an application, an elliptic boundary value problem with cubic nonlinearity, given by 
$\Delta u + \lambda u -u^3+h=0$
on a bounded, smooth domain $\Omega\subset \R^d$ with either Dirichlet or Neumann boundary conditions, is often considered. Here, the forcing term $h$ is a given Hölder continuous function with exponent $\alpha\in(0,1)$. This PDE, equipped with Neumann boundary conditions, has  been studied in \cite{Ruf1990,ruf1992ForcedSecondaryBifur,ruf1995HigherSingularities} in detail. In \cite{Ruf1990}, Ruf showed that for certain parameter values $\lambda$ the elliptic equation has at most three solutions  for arbitrary forcing terms $h$. In the subsequent paper \cite{ruf1992ForcedSecondaryBifur}, it is shown that for other parameter values $\lambda$ a secondary bifurcation occurs. In particular, at least five solutions can occur  for certain forcing terms. Ruf \cite{ruf1995HigherSingularities} also gave a characterization of the solution geometry for the elliptic equation with cubic nonlinearity (and Neumann boundary conditions), while, independently,  the same problem (with Dirichlet boundary conditions) was studied in \cite{Church1993NonlinOperator}. Church and Timourian also derived abstract conditions for the global equivalence of a nonlinear mapping to the fold map \cite{Church1992fold} and the cusp map \cite{Church1993cusp}. A nice historical overview about progress in singularity theory in differential equations is provided in \cite{ruf1997}.} 

{The singularities mentioned above can be understood as singularities of functionals $G \colon E \to E$, where the equation $G(u)=0$ corresponds to the weak form of the PDE. In contrast, we investigate singularities of functionals $S \colon E \to \R$ where the weak formulation of the PDE is recovered as $\D S(u) = 0$, where $\D$ denotes the Fr\'echet derivative. These singularities are related to classical catastrophe theory.}
We will mainly consider the important class of $A$-series singularities which can be defined informally as solutions at which the linearised equation has a one-dimensional kernel. 
However, due to the unified treatment of bifurcations in catastrophe theory, similar methods can be used to detect other high codimensional bifurcations in catastrophe theory such as $D$-series bifurcations. \Cref{fig:AseriesIllus} shows illustrations of the first three $A$-series singularities: the fold ($A_2$), the cusp ($A_3$) and the swallowtail ($A_4$) bifurcation. Further elementary singularities are discussed in \cite{Gilmore1993catastrophe}. Numerical methods to calculate bifurcation points of systems of equations are discussed in \cite{Beyn_1984,Bohmer1993,BOHMER1999277,Fink,GRIEWANK,Hermann,Kunkel,Seydel_2010}.  Fink et al.\ \cite{Fink} present a general mathematical framework for the numerical study of  bifurcation phenomena associated with  parameter-dependent equations. Singularities are directly computed as solutions of a minimally augmented defining system \cite{GRIEWANK,Hermann}. Besides, singularities of functionals with an application to PDEs are discussed in \cite{Kielhoefer_2012}. A numerical experiment for detecting a high codimensional $A$-series bifurcation in a one-dimensional example can be found in \cite[\S 4]{Beyn1984}, for instance. 

The two main approaches to bifurcation of solutions to PDEs which have been considered in the past are generalized Lyapunov-Schmidt reductions \cite{Bohmer1993,BOHMER1999277,Kunkel,Mei2000} and topological methods in the calculus of variations which so far can only access relatively simple bifurcations. The Lyapunov-Schmidt reduction can be used to study solutions to nonlinear equations when the implicit function theorem cannot be applied, and allows the reduction of infinite-dimensional equations in Banach spaces to finite-dimensional equations. However, Lyapunov-Schmidt reductions do not use the variational structure.
This motivates the development of a method for detecting high codimensional $A$-series bifurcations in high-dimensional or even infinite-dimensional phase spaces which makes use of the underlying variational structure. Based on the Infinite-Dimensional Splitting Lemma \cite{Golubitsky1983}, we derive an augmented system for $A$-series singularities and apply it to PDEs by considering their variational formulations. For the numerical implementation, we use the most natural approach and discretise with a variational method. This procedure leads to a numerical method for the detection of high codimensional bifurcations in catastrophe theory which can be applied to a large class of parameter-dependent PDEs. In particular, our original contributions are:
\begin{enumerate}
	\item The use of the Infinite-Dimensional Splitting Lemma.
	\item A unified and simplified treatment of all $A$-series bifurcations.
	\item The presentation in Banach spaces, i.e.\ our results apply both to variational PDEs and their discretisations.
	\item Further simplification for the parameter-dependent semilinear Poisson equation (both in the continuous and discrete setting).
	\item A unified treatment of the continuous problem and its discretisation.
\end{enumerate}

The paper is organised as follows. In \cref{sec:SplittingandBifurTest} we recall an observation by Golubitsky and Marsden \cite{Golubitsky1983} that catastrophe theory, i.e.\ the classification of the bifurcation behaviour of critical points of smooth, real-valued functions on (finite-dimensional) spaces, applies to smooth, nonlinear functionals on Banach spaces.
We use this result to derive explicit bifurcation test equations (also known as determining equations or augmented systems) for all A-series singularities, expressed in terms of the derivatives of the original functional.
Based on these derived augmented systems of equations, we propose a numerical method for finding high codimensional bifurcations in parameter-dependent PDEs. This numerical scheme is illustrated in  \cref{sec:NumExp} where we consider  a Bratu-type boundary value problem as an  of a second order PDE  and detect its  high codimensional bifurcations numerically.  Finally, we conclude in \cref{sec:Conclusion}.

\begin{figure}
\begin{center}
\includegraphics[width=0.17\textheight]{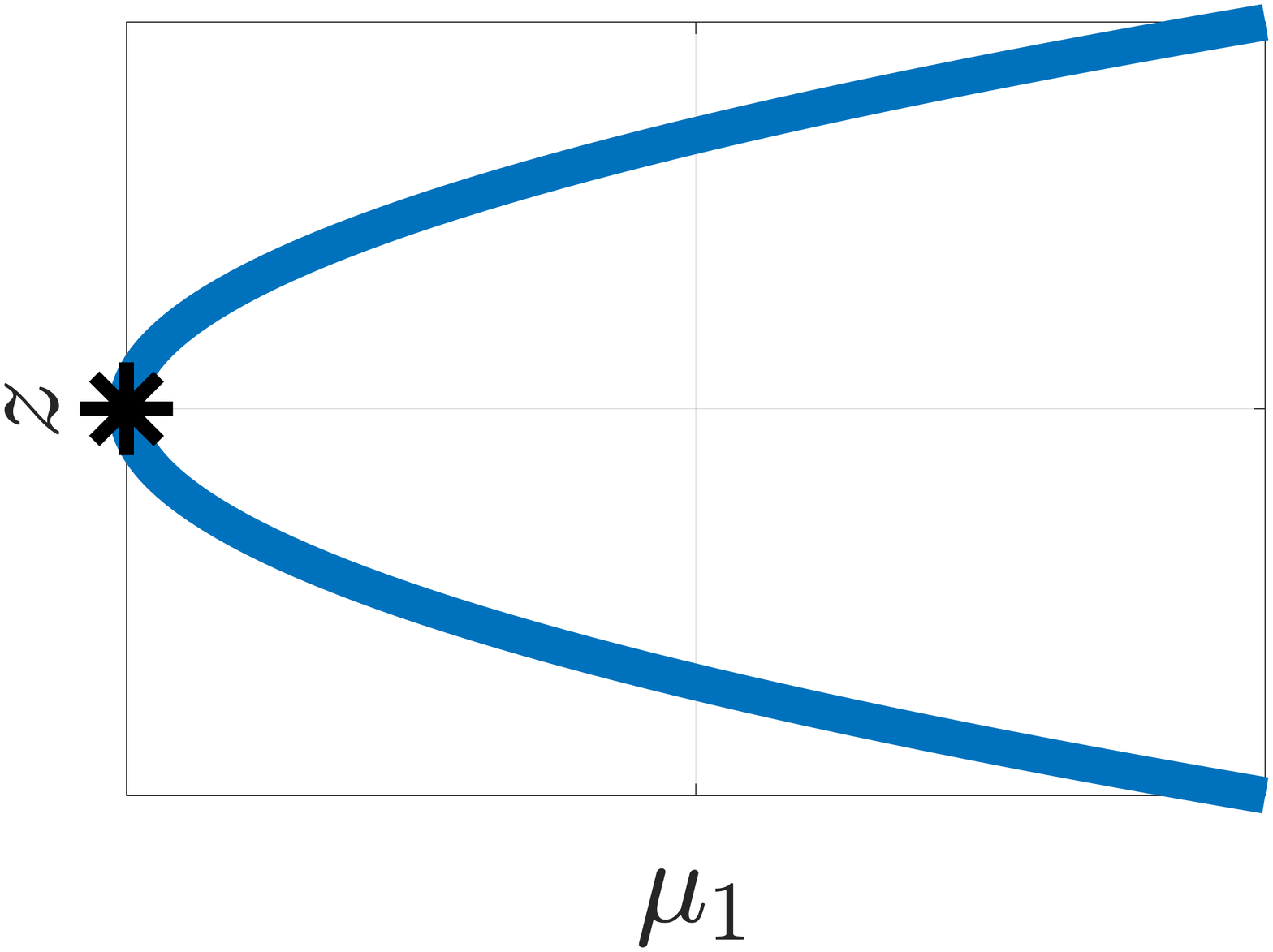}\;
\includegraphics[width=0.16\textheight]{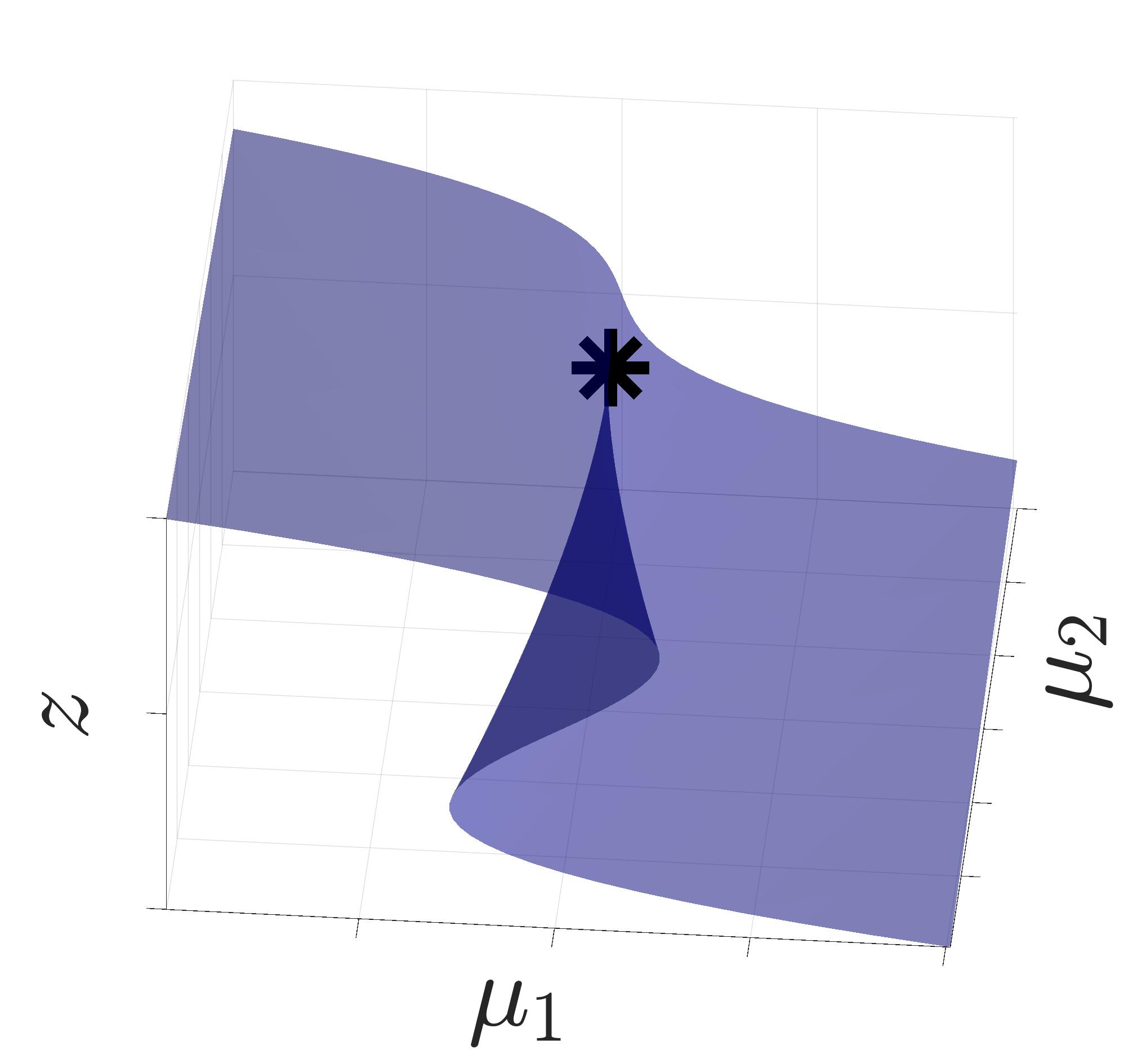}\;
\includegraphics[width=0.16\textheight]{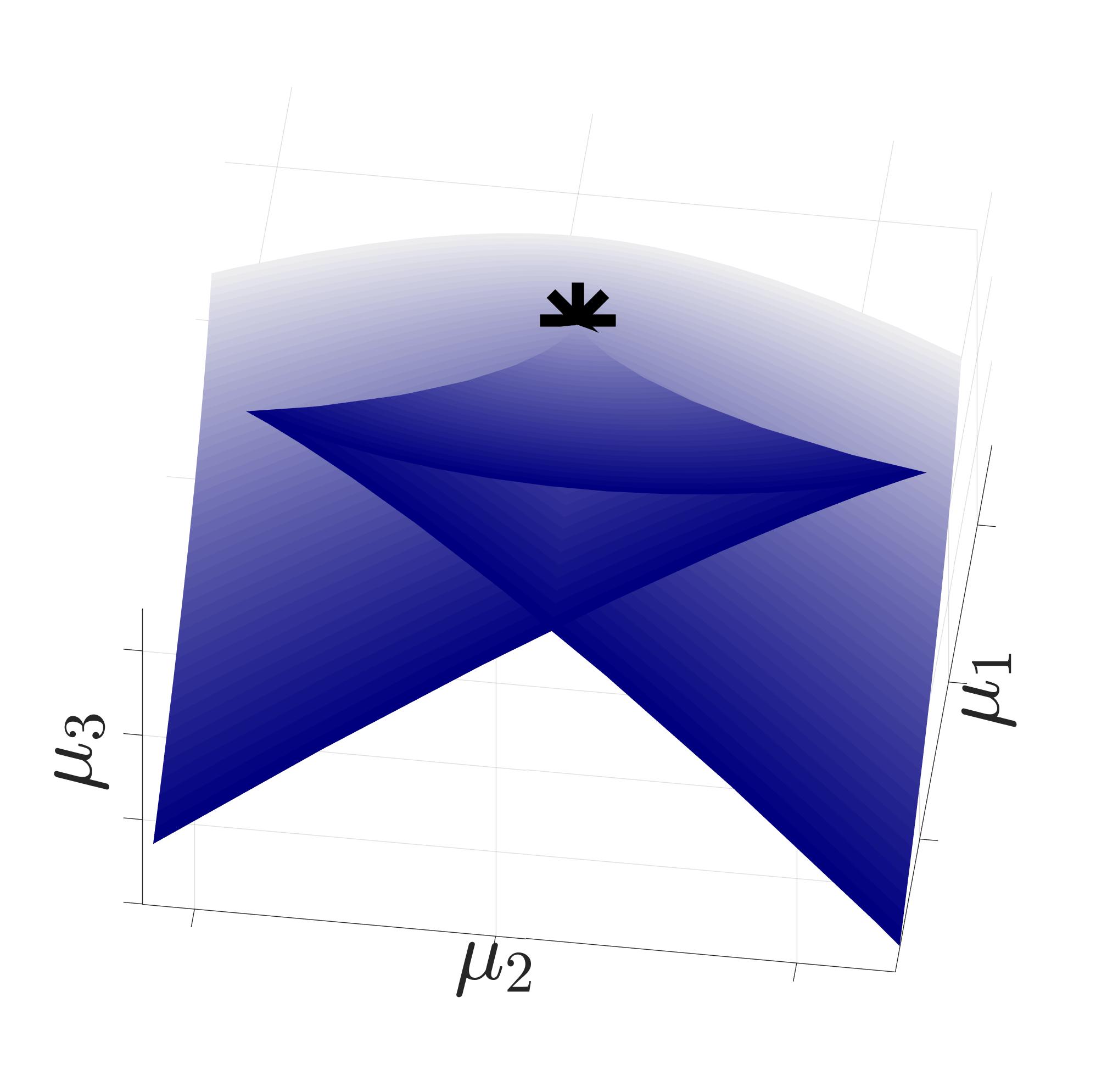}
\end{center}
\caption{The fold ($A_2$), the cusp ($A_3$) and the swallowtail bifurcation ($A_4$) (from left to right). In the diagrams $\mu_1$, $\mu_2$, $\mu_3$ denote parameters and $z$ a dependent. We assume that for a given set of parameters no two solutions have the same $z$ value. In the first diagram we see how two solutions are born out of a fold singularity ($A_2$) denoted by $\ast$. The second diagram shows a surface of solutions over the $\mu_1$, $\mu_2$ plane. For different combinations of parameters $(\mu_1,\mu_2)$ there are 3, 2 or 1 solutions. The point of highest singularity, i.e.\ the cusp point ($A_3$), is denoted by $\ast$. The third plot shows a projection of a bifurcation diagram with a swallowtail point ($A_4$) to the parameter space. Each point in the sheet corresponds to a fold singularity, an intersection of sheets means that two fold singularities happen for the same parameter value but at different points in the phase space, points on the two edges correspond to cusp points. The point where the two edges join with the intersection line is a swallowtail point marked by $\ast$. }\label{fig:AseriesIllus}
\end{figure}

\section{Augmented systems for nonlinear functionals}\label{sec:SplittingandBifurTest}


\subsection{The Splitting Lemma in Banach spaces}\label{sec:SplittingLemmaBanach}

Let $E$ be a Banach space, let $U\subset E$ be an open neighbourhood of the origin and let $S \colon U \to \R$ be a smooth function with $S(0)=0$.
We define the following two assumptions.

\begin{assumption}\label{as:T1}
There exists an inner product $\langle \cdot , \cdot \rangle_E$ on $E$ and a Fredholm operator $T\colon E\to E$ of index 0 such that
\[
\D^2 S(0)(u,v) = \langle Tu , v \rangle_E \quad \text{for all } u,v \in E.
\]
\end{assumption}

In \cref{as:T1} the symbol $\D$ denotes the Fr\'echet derivative. The second order derivative $D^2S(0)$ of $S$ at $0$ is a symmetric bilinear form on $E\times E$. The index-0 Fredholm operator $T$ is symmetric such that $E = \ker T \oplus \rg T$, where $\ker T$ denotes the kernel of $T$ and $\rg T$ the range of $T$. We denote elements $u \in E$ by its components $u = (x,y)$ w.r.t.\ the splitting, i.e.\ $x\in \ker T$ and $y\in \rg T$.

\begin{assumption}\label{as:T2}
There exists a {\em partial gradient} $\nabla_y S \colon U \to \rg T$ with $\nabla_y S(0)=0$ such that
\[
\forall u\in U,\quad \forall v \in \rg T  : \quad \langle \nabla_y S(u),v\rangle_E =\D S(u)v.
\]
\end{assumption}

\begin{theorem}[Infinite-Dimensional Splitting Lemma \cite{Golubitsky1983}]\label{thm:SplittingLemmaInfinte}
If \cref{as:T1} and \cref{as:T2} hold then there exists a fibred change of coordinates $( \bar x, \bar y) = ( x,\eta(x,y))$ on $U$ fixing $(0,0)\in E$ with $\eta \colon U \to \rg T$ such that $\D_y \eta (0,0)= \Id$ and such that $S$ takes the form
\[
S(\bar x, \bar y)= \frac 12 \langle T\bar y,\bar y \rangle + r(\bar x)
\]
for a smooth function $r\colon \mathcal O \to \R$, where $\mathcal O$ is an open neighbourhood of $0 \in \ker T$ and $r(0)=0$, $\D r(0)=0$ and $\D^2 r(0)=0$. 
\end{theorem}

In \cref{thm:SplittingLemmaInfinte} $\D_y \eta$ denotes the Fr\'echet derivative of $\eta$ in the direction of $\rg T$. 
For a discussion of the setting and examples refer to \cite{Buchner1983}. The function $r$ is defined as follows: by the implicit function theorem, there exists a unique, smooth function $F\colon U_{\ker T} \to U_{\rg T}$  for open neighbourhoods  $U_{\ker T}\subset \ker T$ and $U_{\rg T}\subset \rg T$  of the origin in $\ker T$ and $\rg T$, respectively, such that
\begin{equation}\label{eq:defF}
\begin{split}
F(0)=0, \quad \D F(0)=0,\\
\forall x \in U_{\ker T} : \nabla_y S(x,F(x)) = 0.
\end{split}
\end{equation}
The map $r\colon U_{\ker T} \to \R$ is given by
\begin{equation}\label{eq:r}
r(x) = S(x,F(x)).
\end{equation}
We see that critical points of $S$ correspond to critical points of $r$ which is defined on a finite-dimensional space. Singularity theory for $S$ thus reduces to ordinary, finite-dimensional catastrophe theory \cite{Arnold2012,Wassermann1974}. More precisely, to determine the singularities of $S$ it suffices to determine the singularities of $r$.

\subsection{Augmented systems for $A$-series singularities}\label{subsec:AugSysASeries}

$A$-series singularities for a given functional $S$ are defined as follows.

\begin{definition}[$A$-series singularity]\label{def:ASing}
Let $E$ be a Banach space, let $U\subset E$ be a neighbourhood of the origin and let $S \colon U \to \R$ be a smooth function with $S(0)=0$.
Assume that \cref{as:T1} and \cref{as:T2} hold and consider the function $r$ from \eqref{eq:r}. The function $S$ has a singularity of type $A_n$ at 0 if $\ker T$ (with $T$ as in \cref{as:T1}) is one-dimensional and
\[
\left.\frac{\d^k}{\d x^k}\right|_{x=0} r(x) =0,
\quad k=3,\ldots,n,
\qquad \left.\frac{\d^{n+1}}{\d x^{n+1}}\right|_{x=0} r(x) \not = 0.
\]
\end{definition}

\begin{remark}
The singularity $A_2$ is referred to as {\em fold}, $A_3$ as {\em cusp}, $A_4$ as {\em swallowtail} and $A_5$ as {\em butterfly singularity}.
\end{remark}

\begin{remark}
{The Infinite-Dimensional Splitting Lemma (\cref{thm:SplittingLemmaInfinte}) allows us to borrow the notions of catastrophe theory and to define singularities of real-valued functionals $S \colon E \to \R$ fulfilling \cref{as:T1} and \cref{as:T2}. In applications, $\D S(u) (v) = 0\, \forall v \in E$ is the weak formulation of a PDE.
In the literature the gradient structure is typically not exploited for this purpose. Instead the more general class of singularities of functionals $\tilde S \colon E \to \tilde E$ between two Banach spaces $E$ and $\tilde E$ is considered as in \cite{Ambrosetti1972,Berger1985,Berger1988,Lazzeri1987,Church1992fold,Church1993cusp,ruf1995HigherSingularities,Ruf1990,ruf1992ForcedSecondaryBifur,Church1993NonlinOperator}. The general problem class $\tilde S(u)=0$ contains the class of catastrophe problems $\D S(u) = 0$. However, since the class of functions $E \to \tilde E$ is richer, this leads to a different notion of $A$-series singularities since stability properties do not coincide. To illustrate this point, let $\tilde E=E$. A map $\tilde S \colon E \to E$ with a singularity that is persistent under small perturbations $\tilde S + \e \tilde P$ with functions $\tilde P \colon E \to E$ is not necessarily of the form $\tilde S = \D S$. On the other hand a map $S \colon E \to \R$ with a singularity that is persistent under small perturbations $S + \e P$ with functions $P \colon E \to \R$ does not necessarily yield a map $ \tilde S = \D S \colon E \to E$ with a singularity that is persistent under small perturbations with functions of the bigger problem class $\tilde P \colon E \to E$ since perturbations with $\tilde P \not = \D P$ are allowed.
Let us illustrate the different notions of singularities on the cusp singularity. As proved by Whitney, the cusp map $f\colon \R^2 \to \R^2$ with $f(z,w) =(z^3+z w,w)$ is stable. If we plot the $z$-component of solutions $(z,w)$ to the system $f(z,w)=(\mu_1,\mu_2)$ over the $(\mu_1,\mu_2)$-plane then we obtain the plot in the centre of \cref{fig:AseriesIllus}. On the other hand, consider the map $h \colon \R^2 \to \R$ with $h(z,w) = \frac 14 z^4 + \frac 12 w^2$. The map $h$ has a cusp singularity at $(0,0)$. Its universal unfolding in catastrophe theory is given by $h_\mu(z,w) = \frac 14 z^4 + \frac 12 \mu_2 z^2 + \mu_1 z + \frac 12 w^2$. If we plot the $z$-component of solutions $(z,w)$ to the system $\nabla h_\mu(z,w)=(0,0)$ over the $(\mu_1,\mu_2)$-plane then we also obtain the plot in the centre of \cref{fig:AseriesIllus}.
Despite the same visualisation, the cusp map $f$ is not to be confused with the map $\nabla h$. The map $f$ does not have a primitive $g \colon \R^2 \to \R$ with $\nabla g = f$. Moreover, the map $f$ is stable in the class of smooth functions $\R^2 \to \R^2$ while the cusp catastrophe $h\colon \R^2 \to \R$ needs to be unfolded to be stable in the class of smooth functions $\R^2 \to \R$.
In this paper we will investigate the catastrophe setting and exploit the gradient structure for numerical continuation methods.}
\end{remark}

The $k^{\mathrm{th}}$ derivative of $S$ at a point $z \in U$ is a symmetric multi-linear form which we denote by $\D^{(k)}S(z)$. We can interpret the multi-linear form $\D^{(k)}S(z)$ as a linear form on $E^{\otimes k} = \bigotimes_{j=1}^k E$ where $\otimes$ denotes the tensor product. As a shorthand we define $S^{(k)} := \D^{(k)}S$ and $S_0^{(k)} := \D^{(k)}S(0)$.

We can express the condition for a function $S$ to have an $A$-series singularity in terms of (Fr\'echet)-derivatives of $S$. For this we define the multi-index set $\mathcal{J}_k^{n}$ for $n,k \in \N$ with $k \le n$ as

\begin{align}\label{eq:BellPolyIndexSet}
\left\{j=(j_1,\ldots,j_{n-k+1})\colon \; \;
j_l \in \N \cup\{0\}, \; \; 
\sum_{l=1}^{n-k+1}j_l = k, \; \;
\sum_{l=1}^{n-k+1}l \cdot j_l = n\right\}.
\end{align}
Moreover, for a multi-index $j \in \mathcal{J}_k^{n}$ we define $j! := j_1 ! j_2 ! \cdots j_{n-k+1}!$.

\begin{theorem}\label{thm:ASeriesDetect}
Let $E$ be a Banach space and consider an inner product $\langle \cdot , \cdot \rangle_E$ on $E$. Let $U\subset E$ be a neighbourhood of the origin and let $S \colon U \to \R$ be a smooth function with $S(0)=0$.
%
Consider the following algorithm consisting of a sequence of tests. The algorithm terminates and returns the current value of the integer $n$ if a test fails. In the algorithm $F$ is considered as a function symbol of an unknown smooth map $I \to E$, where $I$ is a small open interval in $\R$ containing 0 and with $F(0)=F^{(0)}(0)=0$, $F^{(1)}(0)=0$.

\begin{itemize}
\item
Set $n=1$. Test $S_0^{(1)}=0$. 
\item
Set $n=2$. Test whether the kernel of 
$a \mapsto S_0^{(2)}(a,\cdot ) \in E^\ast$
is 1-dimensional.
\item
Set $n=3$. Select an element $\alpha\in E\setminus\{0\}$. Test $S^{(3)}(\alpha^{\otimes 3}) = 0$.
\item
Loop through the following two steps.
\begin{enumerate}
\item
Set $n:=n+1$. Determine $F^{(n-2)}(0) \in \{{s} \alpha\, | \, {s}  \in \R \}^{\bot_E}$ using
\begin{equation*}
\forall \xi \in \{{s}  \alpha\, | \, {s}  \in \R \}^{\bot_E} :
\end{equation*}
\begin{equation*}
\hspace{-1cm}
\sum_{k=1}^{n-2 }
{\sum_{j\in\mathcal{J}_k^{n-2}} \frac{ (n-2) !}{{ j!}}}
S^{(k+1)}_0
\left(\xi \otimes \bigotimes_{l=1}^{n-k-1}
\frac {\left( \left.\frac{\d^{l}}{\d s^{l}}\right|_{s=0}(s\alpha + F(s))\right)^{\otimes j_l}} {{(l!)}^{j_l}}
\right)=0.
\end{equation*}

\item
Test 
\[ \sum_{k=1}^n 
\sum_{{j\in \mathcal{J}_k^{n}}} \frac{n!}{{{ j!}}} 
S^{(k)}_0\left(\bigotimes_{l=1}^{n-k+1} \frac {\left(\left.\frac{\d^{l}}{\d s^{l}}\right|_{s=0}(s\alpha + F(s))\right)^{\otimes j_l}} {{(l!)}^{j_l}}\right)=0.\]
\end{enumerate}
\end{itemize}
The following statements hold true.
\begin{itemize}
\item
The algorithm returns $n=1$ if and only if $0$ is not a critical point of $S$.
\item
The algorithm returns $n=2$ or does not terminate if and only if $0$ is a critical point of $S$ but $S$ does not have an $A$-series singularity at $0$.
\item
If  \cref{as:T1} and \cref{as:T2} are satisfied w.r.t.\ $\langle \cdot , \cdot \rangle_E$ then the algorithm returns $n \ge 3$ if and only if $S$ has a singularity of type $A_{n-1}$.
\end{itemize}
\end{theorem}

Before proving the theorem, let us formulate some corollaries which illustrate how the conditions in \cref{thm:ASeriesDetect} simplify for small values of $n$.

\begin{definition}
We say that a functional $S$ fulfilling \cref{as:T1} and \cref{as:T2} has a singularity {\em of type at least $A_n$} if
\begin{itemize}
\item
it has a singularity of type $A_N$ with $N \ge n$ or
\item
the algorithm in \cref{thm:ASeriesDetect} does not terminate.
\end{itemize}
\end{definition}

\begin{corollary}[Fold ($A_2$)]\label{prop:fold}
Let $E$ be a Banach space and $S$ be a (nonlinear) smooth functional defined on an open neighbourhood of $0 \in E$. Assume that \cref{as:T1} and \cref{as:T2} hold for a Fredholm operator $T$ such that
\begin{equation}\label{eq:foldT}
\ker T = \mathrm{span}_\R \{\alpha\}
\end{equation}
is satisfied for a non-trivial element $\alpha \in E$. Then $S$ has a singularity of type at least $A_2$ (fold singularity).
\end{corollary}

\begin{corollary}[Cusp ($A_3$)]\label{prop:HilbertCusp}
Let $E$ be a Banach space and $S$ be a smooth functional defined on an open neighbourhood of $0 \in E$. Assume that \cref{as:T1} and \cref{as:T2} hold for a Fredholm operator $T$ and $\alpha \in E \setminus \{0\}$ such that \eqref{eq:foldT} holds, i.e.\ $S$ has a singularity of type at least $A_2$. The functional $S$ has a singularity of type at least $A_3$ if and only if
\begin{equation}\label{eq:CuspEasy}
S^{(3)}_0(\alpha,\alpha,\alpha)=0.
\end{equation}
\end{corollary}

\begin{corollary}[Swallowtail ($A_4$)]\label{prop:SW}
\label{prop:HilbertSW}
Let $E$ be a Banach space and $S$ be a smooth functional defined on an open neighbourhood of $0 \in E$. Assume that \cref{as:T1} and \cref{as:T2} hold for a Fredholm operator $T$ and $\alpha \in E \setminus \{0\}$ such that \eqref{eq:foldT} and \eqref{eq:CuspEasy} are satisfied, i.e.\ $S$ has a singularity of type at least $A_3$. The functional $S$ has a singularity of type at least $A_4$ if and only if
\begin{equation}\label{eq:v}
S^{(2)}_0(v,\xi) = - S^{(3)}_0(\alpha,\alpha,\xi) \quad \forall \xi \in E
\end{equation}
is solvable for $v \in E$ and 
\begin{equation}\label{eq:SWCond}
     S^{(4)}_0(\alpha,\alpha,\alpha,\alpha)
-3 S^{(2)}_0(v,v)
= 0.
\end{equation}
\end{corollary}

\begin{corollary}[Butterfly ($A_5$)]\label{prop:butterfly}
Let $E$ be a Banach space and $S$ be a smooth functional defined on an open neighbourhood of $0 \in E$.
Assume that \cref{as:T1} and \cref{as:T2} hold for a Fredholm operator $T$ and $\alpha \in E \setminus \{0\}$ such that \eqref{eq:foldT} and \eqref{eq:CuspEasy} are satisfied.
Furthermore, assume that \eqref{eq:v} holds for some $v \in E$ and \eqref{eq:SWCond} is satisfied, i.e.\ the functional $S$ has a singularity of type at least $A_4$.
The functional $S$ has a singularity of type at least $A_5$ if and only if
\begin{equation}\label{eq:w}
S^{(2)}_0(w,\xi) = - S^{(4)}_0 (\alpha,\alpha,\alpha,\xi)-3S^{(3)}_0(\alpha,v,\xi) \quad \forall \xi \in E
\end{equation}
is solvable for $w \in E$ and 
\[
S^{(5)}_0 (\alpha^{\otimes 5}) - 15 S^{(3)}_0 (\alpha ,v,v) + 10 S^{(3)}_0 (\alpha,\alpha,w)=0.
\]
\end{corollary}

\begin{remark}\label{rem:uniquenessnotnecessary}
Notice that we do not need to require $v \in \rg T$ in \eqref{eq:v} or $w \in \rg T$ in \eqref{eq:w} since $S^{(2)}_0(w+t_1\alpha,v+t_2 \alpha) = S_0^{(2)}(w,v)$ and $S^{(3)}_0(\alpha,v+t_1\alpha,\xi+t_2 \alpha) = S^{(3)}_0(\alpha,v,\xi)$ {for $\alpha\in E$ satisfying} \eqref{eq:foldT} and \eqref{eq:CuspEasy} such that the test equations are defined invariantly.
\end{remark}

\begin{remark}
The equations in the loop section of the algorithm presented in \cref{thm:ASeriesDetect} can be obtained from the ${(n-2)^{\mathrm{th}}}$ and $n^{\mathrm{th}}$ {\em complete exponential Bell polynomials} as we will see from \cref{prop:DerivativesBellPoly}, \cref{lem:determineF0} and \cref{rem:detF0Bell}.
\end{remark}

\begin{definition}[Complete exponential Bell polynomial]
The $n^{\mathrm{th}}$ {\em complete exponential Bell polynomial} is given as
\begin{align}\label{eq:bellpolynomial}
B_n(x_1,\ldots,x_n) 
= \sum_{k=1}^n \sum_{{j\in\mathcal{J}_k^n}} \frac{n!}{j!} \prod_{l=1}^{n-k+1} \left(\frac {x_l} {l!} \right)^{j_l},
\end{align}
with the multi-index set $\mathcal{J}_k^n$ as in \eqref{eq:BellPolyIndexSet} and $j! = j_1 ! j_2 ! \ldots j_{n-k+1}!$ for $j \in \mathcal{J}_k^n$.
\end{definition}

The first five complete exponential Bell polynomials are given by
\begin{align}\label{eq:BellPolysMonForm}\begin{split}
B_{0}={}&1,\\
B_{1}(x_{1})={}&x_{1},\\
B_{2}(x_{1},x_{2})={}&x_{1}^{2}+x_{2},\\
B_{3}(x_{1},x_{2},x_{3})={}&x_{1}^{3}+3x_{1}x_{2}+x_{3},\\
B_{4}(x_{1},x_{2},x_{3},x_{4})={}&x_{1}^{4}+6x_{1}^{2}x_{2}+4x_{1}x_{3}+3x_{2}^{2}+x_{4},\\B_{5}(x_{1},x_{2},x_{3},x_{4},x_{5})={}&x_{1}^{5}+10x_{2}x_{1}^{3}+15x_{2}^{2}x_{1}+10x_{3}x_{1}^{2}+10x_{3}x_{2}+5x_{4}x_{1}+x_{5}.
\end{split}
\end{align}

\begin{remark}
Complete exponential Bell polynomials appear as coefficients in the following formal power series.
\[
\exp\left( \sum_{k=1}^\infty \frac{x_k}{k!} y^k \right) = \sum_{n=0}^\infty \frac 1 {n!} B_n(x_1,\ldots,x_n) y^n.
\]
Moreover, the $n^{\mathrm{th}}$ complete exponential Bell polynomial encodes information on the number of ways a set containing $n$ elements can be partitioned into non-empty, disjoint subsets. For example we can read off from
\[
B_{4}(x_{1},x_{2},x_{3},x_{4})={}x_{1}^{4}+6x_{1}^{2}x_{2}+4x_{1}x_{3}+3x_{2}^{2}+x_{4}
\]
that there is
\begin{itemize}
\item
1 partition consisting of 4 sets of cardinality 1,
\item
6 partitions into 3 sets of which 2 have cardinality 1 and the other one has cardinality 2,
\item
4 partitions into 2 sets of which 1 has cardinality 1 and the other one has cardinality 3,
\item
3 partitions into 2 sets of cardinality 2,
\item
and 1 partition consisting of 1 set of cardinality 4.
\end{itemize}
(See \cite{BellPartitionPolys,brualdi2004introductory}, for instance).
\end{remark}

Let us prepare the proof of \cref{thm:ASeriesDetect} with two lemmas.

\begin{lemma}\label{prop:DerivativesBellPoly}
Let $E$ be a Banach space, $U \subset E$ an open subset and $\alpha \in E$. Consider smooth functions $S \colon U \to \R$ and $F \colon I \to E$, where $I$ is an open interval $I \subset \R$ such that $r(s):=S(s\alpha+F(s))$ is defined on $I$.
For $n \in \N$ we have
\[
r^{(n)}(s) = B_n(\alpha + F'(s),F''(s),\ldots,F^{(n)}(s)).
\]
On the right hand side of the equation multiplications are interpreted as tensor products. Moreover, the symbol ``+'' is replaced by ``$ + S^{(\mathrm{degree})}(s\alpha+F(s))$'', where $\mathrm{degree}$ is the count of factors in the tensor product to which $S^{(\mathrm{degree})}(s\alpha+F(s))$ is applied. In other words 
\[r^{(n)}(s) = 
\sum_{k=1}^n 
\sum_{{j\in \mathcal{J}_k^{n}}} \frac{n!}{j!} 
S^{(k)}(s\alpha+F(s))\left(
\bigotimes_{l=1}^{n-k+1} 
\frac {\left(\frac{\d^{l}}{\d s^{l}}(s\alpha + F(s))\right)^{\otimes j_l}} {{(l!)}^{j_l}}
\right),
\]
where $j=(j_1,\ldots,j_{n-k+1})\in \mathcal{J}_k^n$ is defined in \eqref{eq:BellPolyIndexSet} and $j!=j_1 ! j_2 ! \ldots j_{n-k+1}!$ for a multi-index $j \in \mathcal{J}_k^n$.
\end{lemma}

\begin{corollary}\label{cor:derr}
In the setting of \cref{prop:DerivativesBellPoly},
if $0 \in U$ and $S_0=0$, $S_0^{(1)}=0$, $S_0^{(2)}(\alpha,\xi)=0$ for all $\xi \in E$ and $F(0)=F'(0)=0$ then the first five derivatives of $r$ evaluated at 0 are given by
\begin{align*}
r(0)&=0\\
r'(0)&=0\\
r''(0)&=0\\
r'''(0) &= S^{(3)}_0 (\alpha^{\otimes 3} ) \\
r''''(0) &= S^{(4)}_0 (\alpha^{\otimes 4} ) + 6 S^{(3)}_0 (\alpha^{\otimes 2} \otimes F''(0))
	   + 3 S^{(2)}_0 (F''(0)^{\otimes 2}) \\
r'''''(0)&= S^{(5)}_0 (\alpha^{\otimes 5}) 
	   + 10 S^{(4)}_0 (\alpha^{\otimes 3} \otimes F''(0))
   	   + 10 S^{(3)}_0 (\alpha^{\otimes 2} \otimes F'''(0))\\
   	   &+ 15 S^{(3)}_0 (\alpha \otimes (F''(0))^{\otimes 2})
   	   + 10 S^{(2)}_0 (F''(0)\otimes F'''(0) ).
\end{align*}
\end{corollary}

\begin{proof}[Proof of \cref{{prop:DerivativesBellPoly}}]
The statement is an extension of Fa\`a di Bruno's formula \cite[\S 12.3]{andrews_1984} from $E=\R$ to arbitrary Banach spaces $E$.
Since the chain rule for differentiation is valid for Fr\'echet derivatives \cite[Thm.\ 2.47]{Penot2013}, the proof is analogous.
\end{proof}

As we see explicitly for $i\le 5$ in \cref{cor:derr}, to determine $r^{(i)}(0)$ the $i-2$-jet of $F$ is required at 0 in the setting of \cref{cor:derr}. This holds in general.

\begin{corollary}\label{cor:ReqJet}
In the setting of \cref{cor:derr} the value $r^{(i)}(0)$ can be calculated from the $i-2$-jet of $F$ where $i \in \N$ with $i \ge 2$. We can write
\[
r^{(i)}(0) = B_i(\alpha ,F''(0),\ldots,F^{(i-2)}(0),c_2,c_1),
\]
where $c_1,c_2$ are any constants. On the right hand side of the equation multiplications are interpreted as tensor products. Moreover, the symbol ``+'' is replaced by ``$ + S^{(\mathrm{degree})}_0$'', where $\mathrm{degree}$ is the count of factors in the tensor product to which $S^{(\mathrm{degree})}_0$ is applied.
\end{corollary}

\begin{proof}
We use the combinatorial interpretation of Bell polynomials and \cref{prop:DerivativesBellPoly}.
If a partition of an $i$-set contains a subset with $i-1$ elements then there must be exactly one more non-empty subset containing one element to form a valid partition. Therefore, $F^{(i-1)}(0) = \left.\frac{\d^{i-1}}{\d s^{i-1}}\right|_0(s\alpha + F(s)) $ only occurs together with $\alpha = \left.\frac{\d^{1}}{\d s^{1}}\right|_0(s\alpha + F(s))$ as $i F^{(i-1)}(0) \otimes \alpha$. This terms becomes an input argument of $S^{(2)}_0$ and, therefore, vanishes by the choice of $\alpha$. If a partition of an $i$-set contains a subset with $i$ elements then there cannot be another non-empty subset in the partition. Therefore, $F^{(i)}(0) = \left.\frac{\d^{i}}{\d s^{i}}\right|_0(s\alpha + F(s))$ becomes an input argument of $S^{(1)}_0$ which is zero.
\end{proof}

%

\begin{lemma}\label{lem:determineF0}
Let $E$ be a Banach space and let $U$ be an open neighbourhood of $0 \in E$. Consider a smooth function $S \colon U \to \R$ with $S(0)=0$ such that \cref{as:T1} and \cref{as:T2} hold for a Fredholm operator $T$ with
\begin{equation*}
\ker T = \mathrm{span}_\R \{\alpha\}
\end{equation*}
for a non-trivial element $\alpha \in E$. There exists an open interval $I$ containing 0 and a unique function $F \colon I \to \rg T$ with $F(0)=0$, $F'(0)=0$ s.t.\
\begin{equation}
\label{eq:DefFinLemma}
\forall s \in I, \; \; \forall \xi \in \rg T :\quad  S^{(1)}(s\alpha + F(s))(\xi)=0.
\end{equation}
Moreover, all derivatives $F^{(n)}(0)$ with $n \ge 2$ can be obtained successively from
\begin{align}\label{eq:recrelF0}
\forall \xi \in \rg T : 0&=\sum_{k=1}^n 
\sum_{{j\in \mathcal{J}_k^n}} \frac{n!}{ j!} 
S^{(k+1)}_0\left(\xi \otimes \bigotimes_{l=1}^{n-k+1} \frac {\left(\left.\frac{\d^{l}}{\d s^{l}}\right|_{s=0}(s\alpha + F(s))\right)^{\otimes j_l}} {{(l!)}^{j_l}}\right),
\end{align}
where $j=(j_1,\ldots,j_{n-k+1})\in \mathcal{J}_k^n$ as defined in \eqref{eq:BellPolyIndexSet} and ${ j!} = j_1 ! j_2 ! \ldots j_{n-k+1}!$.
\end{lemma}

\begin{remark}\label{rem:detF0Bell}
Relation \eqref{eq:recrelF0} can be expressed as
\[
0 = B_n(\alpha,F''(0),\ldots,F^{(n)}(0)),
\]
where multiplications in the definition of the Bell polynomial \eqref{eq:bellpolynomial} are interpreted as tensor products. Furthermore, the symbol ``+'' is replaced by ``$ + S^{(\mathrm{degree})}(s\alpha+F(s)) \xi \otimes $'' and parenthesis may be added around the input arguments of the form $S^{(\mathrm{degree})}(s\alpha+F(s))$.
\end{remark}

\begin{corollary}\label{cor:implDerr}
The relations for $n=2,3,4,5$ of \cref{lem:determineF0} read
\begin{align*}
0&= S^{(3)}_0(\xi \otimes  \alpha^{\otimes 2}) + S^{(2)}_0(\xi \otimes  F''(0))\\
0&= S^{(4)}_0(\xi \otimes  \alpha^{\otimes 3}) + 3 S^{(3)}_0(\xi \otimes \alpha \otimes F''(0))+ S^{(2)}_0(\xi \otimes F'''(0))\\
0&= S^{(5)}_0(\xi \otimes \alpha^{\otimes 4}) 
	  + 6 S^{(4)}_0(\xi \otimes \alpha^{\otimes 2} \otimes F''(0))
	  + 3 S^{(3)}_0(\xi \otimes F''(0)^{\otimes 2} )\\
	   &+ 4 S^{(3)}_0(\xi \otimes \alpha \otimes F'''(0))
	   + S^{(2)}_0(\xi \otimes F''''(0))\\
0      &= S_0^{(6)}(\xi \otimes  \alpha^{\otimes 5} )
	  + 10 S^{(5)}_0(\xi \otimes \alpha^{\otimes 3} \otimes F''(0))\\
   	  &+ 10 S^{(4)}_0(\xi \otimes \alpha^{\otimes 2} \otimes F'''(0))
   	   + 15 S^{(4)}_0(\xi \otimes \alpha \otimes F''(0)^{\otimes 2})\\
   	   &+ 10 S^{(3)}_0(\xi \otimes F''(0)\otimes F'''(0) )
	   + 5 S^{(3)}_0(\xi \otimes \alpha \otimes F''''(0))\\
	   &+ S^2_0(\xi \otimes F'''''(0)).
\end{align*}
\end{corollary}

\begin{proof}[Proof of \cref{lem:determineF0}]
The operator $T$ defines an isomorphism on $\rg T$. Thus, the implicit function theorem applies to $\nabla_y S$ and together with \eqref{eq:DefFinLemma} provides the existence and uniqueness of $F$ with $F(0)=0$ and $F'(0)=0$ in analogy to the proof of \cref{thm:SplittingLemmaInfinte} which can be found in \cite{Golubitsky1983}.
%
%
Let $\xi \in \rg T$. Differentiating 
\begin{equation}\label{eq:DefFNew}
S^{(1)}(s \alpha +F(s) )(\xi) =0\end{equation}
repeatedly with respect to $s$ we obtain the following relations.



\begin{align*}
0&=S^{(1)}(s\alpha+F(s))(\xi)\\
0&=S^{(2)} (s\alpha+F(s)) (\xi \otimes (\alpha + F'(s)))\\
0&= S^{(3)} (s\alpha+F(s)) (\xi \otimes  (\alpha + F'(s))^{\otimes 2}) 
	  + S^{(2)} (s\alpha+F(s)) (\xi \otimes  F''(s))\\
0&= S^{(4)} (x\alpha+F(x)) (\xi \otimes  (\alpha + F'(x))^{\otimes 3}) \\
		&+ 3 S^{(3)} (x\alpha+F(x)) (\xi \otimes (\alpha+F'(x)) \otimes F''(x))\\
		&+ S^{(2)} (x\alpha+F(x)) (\xi \otimes F'''(x))\\
0&=\ldots
\end{align*}
We encounter the same combinatorical relations as in \cref{prop:DerivativesBellPoly} such that differentiating \eqref{eq:DefFNew} $n$ times gives
\[
0 = B_n(\alpha + F'(s),F''(s),\ldots,F^{(n)}(s)),
\]
where multiplications in the equation above are interpreted as tensor products and ``$+$'' is replaced by ``$+S^{(\mathrm{degree})}(s\alpha+F(s)) \xi \otimes $. One may add parenthesis around the input arguments of the form $S^{(\mathrm{degree})}(s\alpha+F(s))$.
In other words the relations are given by
\[0=
\sum_{k=1}^n 
\sum_{j\in \mathcal J^k_n} \frac{n!}{ j!} 
S^{(k+1)}(s\alpha+F(s))\left(\xi \otimes \bigotimes_{l=1}^{n-k+1} \frac {\left(\frac{\d^{l}}{\d s^{l}}(s\alpha + F(s)) \right)^{\otimes j_l}} {{(l!)}^{j_l}}\right).
\]
An evaluation at $s=0$ yields the claimed formula. In the $n^{\mathrm{th}}$ step of differentiation the term $F^{(n)}(0)$ only occurs as an input argument of $S^2_0$ and not elsewhere. Since the symmetric operator $T$ restricted to $\rg T$ is an isomorphism on $\rg T$ this successively determines all derivatives of $F$ at $0$. 
\end{proof}

\begin{proof}[Proof of \cref{thm:ASeriesDetect} and its corollaries]
The first two statements of the theorem follow by definition.
Assume that \cref{as:T1} and \cref{as:T2} hold for an operator $T$ with 1-dimensional kernel. Let $\alpha \in \ker\, T\setminus \{0\}$.
To analyse the singularities of $S$ it suffices to analyse the singularities of the function $r\colon U_{\ker} \to \R$ provided by \cref{thm:SplittingLemmaInfinte}.
Using the identification $\R \xrightarrow{\sim} \ker T$, $s \mapsto s \alpha$, we identify $U_{\ker}$ with an open interval $I$ containing 0 and obtain $r \colon I \to \R$.
The function $r$ has the form $r(s) = S(s \alpha + F(s))$ for a smooth function $F \colon I \to \R$. The functional $S$ has a singularity of type $A_m$ at 0 if and only if $r^{(k)}(0)=0$ for all $k \le m$ and $r^{(m+1)}(0)\not =0$.


The algorithm presented in the statement of \cref{thm:ASeriesDetect} consists of a sequence of tests and a variable $n$ acts as a counter. If the state of the variable $n$ is $k$ then the test in the algorithm corresponds to testing $r^{(k)}(0)=0$. This can be seen from \cref{prop:DerivativesBellPoly}. (The formula for $r^{(k)}(0)$ is related to the $k^{\mathrm{th}}$ complete exponential Bell polynomial.)
To evaluate $r^{(k)}(0)$ the $k-2$-jet of $F$ is required as observed in \cref{cor:ReqJet}. 
If the state of the variable $n$ is $k \ge 4$ then $F^{(k-2)}(0)$ gets determined in the algorithm just before $r^{(k)}(0)=0$ is tested. \Cref{lem:determineF0} justifies that the algorithm can determine $F^{(k-2)}(0)$ via the given formula (which is related to the $(k-2)^{\mathrm{th}}$ complete exponential Bell polynomial) if all values $F^{(i)}(0)$ are defined for $i \le k-3$. The values $F(0)$ and $F'(0)$ are set to be 0.
The theorem follows by induction.


\Cref{prop:fold,prop:HilbertCusp,prop:SW,prop:butterfly} follow from \cref{cor:derr,cor:implDerr}.
In \cref{prop:SW} and \cref{prop:butterfly} the bifurcation test equations $r^{(4)}(0)=0$ and $r^{(5)}(0)=0$ have been simplified using \eqref{eq:v} and \eqref{eq:w}.
\end{proof}





\begin{prop}\label{prop:signA}
Let $E$ be a Banach space and $S$ be a real-valued functional defined on an open neighbourhood of $0 \in E$ such that \cref{as:T1} and \cref{as:T2} hold. Consider the function $r$ of the proof of \cref{thm:ASeriesDetect}, whose derivatives are bifurcation test equations. 
If $S$ has a singularity of type $A_{2k+1}$ with $k \in \N$ at $0$ then the signature of $r^{(2n+2)}(0)$ is well-defined.
\end{prop}

\begin{proof}
The statement follows from the classification of singularities up to {\em right-equivalence} in catastrophe theory \cite{lu1976singularity} or can be deduced from our considerations as follows. The map $r(s) = S(s \alpha + F(s))$ from the proof of \cref{thm:ASeriesDetect} is defined uniquely up to the choice of $\alpha$, where $\alpha$ is as in \cref{as:T1}. 
The determining equations for the jet of $F$ at 0 and the bifurcation test equations $r^{(j)}(0)$ ($j \in \N$) are related to Bell polynomials by \cref{prop:DerivativesBellPoly,lem:determineF0,rem:detF0Bell}.
\begin{itemize}
\item
In any partition of an even amount of elements there must be an even number (or none) of subsets with odd cardinality.
\item
In any partition of an odd amount of elements there must be an odd number of subsets with odd cardinality.
\end{itemize}
Using the two combinatorial observations above we see inductively that the signatures of the derivatives $F^{(2j)}(0)$ are defined invariantly of $\alpha$, all derivatives $F^{(2j+1)}(0)$ change to $-F^{(2j+1)}(0)$ ($j \in \N$) as $\alpha \mapsto -\alpha$ and can conclude that the signature of $r^{(2k+2)}(0)$ is well-defined.
\end{proof}

\begin{remark}\label{rem:PosNegDeterm}
The signature considered in \cref{prop:signA} occurs in catastrophe theory if a classification of singularities up to {\em right-equivalence} is considered \cite{lu1976singularity}. The singularities $A_{2k+1}$ do not have a signature.
If the algorithm in \cref{thm:ASeriesDetect} returns $n=2k+2$ with $k \in \N$ and \cref{as:T1} and \cref{as:T2} are satisfied then the singularity $A_{2k+1}$ is of the positive type if and only if the last test equation (which corresponds to the $n^{\mathrm{th}}$ complete exponential Bell polynomial) is positive. Otherwise the singularity is of the negative type.
\end{remark}

\begin{remark}\label{rem:PosNegMerge}
If the functional $S$ has $2k+1$ parameters ($k \in \N$) then, under non-degeneracy conditions, $A_{2k+1}$ singularities occur as 1-parameter families (by the implicit function theorem).
If two branches of $A_{2k+1}$ singularities merge in a $A_{2k+2}$ singularity then one consists of singularities of the positive type and the other one of singularities of the negative type.
This is because the bifurcation test equation $r^{(2k+2)}$, which determines the signs of the $A_{2k+1}$ singularities, must have a non-degenerate zero at the $A_{2k+2}$ singularity, i.e.\ its graph intersects the axis of abscissas transversally.
\end{remark}

\begin{remark}
\Cref{rem:PosNegMerge} applies to the fold bifurcation $A_2$ as well.
In the finite-dimensional case the signature of a solution $z$ can be obtained as the sign of the determinant\footnote{which does not depend on the choice of basis since $\D S(z)=0$} of the Hessian matrix $S^{(2)}_z$ of $S$ at $z$. In a numerical computation the sign can be determined by performing an LU-decomposition of $S^{(2)}_z$ without pivoting and counting whether the number of positive signs on the diagonal of $U$ is even or odd.
Keeping track of the signatures of solutions, cusps,\ldots, provides information on which ones may be able to meet in a bifurcation.
\end{remark}

\section{Example: semilinear Poisson equation}\label{sec:NumExp}


We will exemplify how the augmented systems derived in \cref{subsec:AugSysASeries} can be applied to PDEs. For this, we consider a second order, semilinear PDE describing the steady state solutions in a reaction-diffusion process \cite{Mei2000Ch1}. 
First, we will justify that the theory presented in \cref{sec:SplittingLemmaBanach} applies and write down the continuous recognition equations. 
By a concrete numerical example we will show how augmented systems can be employed in continuation methods to find high codimensional singularities.


\subsection{Setup}\label{subsec:Setup}

For a smooth function $f \colon \R \times \R^s  \to \R$ we consider the homogeneous Dirichlet problem

\begin{equation}\label{eq:PDE}
\left\{\begin{aligned}
\Delta u + f(u,\lambda) &= 0\\
u|_{\p \Omega} &=0
\end{aligned}\right.
\end{equation}
on {an open and bounded domain $\Omega \subset \R^d$ with boundary $\p \Omega$ of class $\mathcal C^k$, where $k >  \frac d2 +1$.} We denote the standard volume form on $\Omega$ by $\d {\bf x}$.
%
%
%
In analogy to \cite[Example 7]{Buchner1983} we consider the following setting {which will allow us to employ the Infinite-Dimensional Splitting Lemma. See \cite{ruf1995HigherSingularities}, for instance, for an alternative treatment.} 
{The Sobolev space $H^k(\Omega)$ is compactly embedded into $\mathcal{C}^{1}(\overline{\Omega})$ \cite[Thm 6.2]{SobolevSpacesAdams}. Consider the Hilbert space $E=H_0^1(\Omega) \cap H^k(\Omega)$ with the structure inherited from $H^k(\Omega)$ \cite{SobolevSpacesAdams,LaxFuncAna}.}
{Let $\bar f$ be s.t.\ $\frac{\p}{\p t} \bar f (t,\lambda) = f(t,\lambda)$ and consider the non-linear functional $S\colon E \times \R^s \to \R$ defined as
\[
S(u,\lambda) = \int_{\Omega} \left( - \frac 12 \langle \nabla u, \nabla u \rangle + \bar f(u,\lambda)\right) \d {\bf x}.
\]
The Fr\'echet derivatives of $S$ in the directions $v_1,v_2,\ldots \in E$ exist and are given as
\begin{align*}
S^{(1)}(u,\lambda)(v_1) &= \int_{\Omega} \left( -  \langle \nabla u, \nabla v_1 \rangle + f(u,\lambda) v_1\right) \d {\bf x}\\
S^{(2)}(u,\lambda)(v_1,v_2) &= \int_{\Omega} \left( -  \langle \nabla v_1, \nabla v_2 \rangle + f'(u,\lambda) v_1 v_2\right) \d {\bf x}\\
S^{(l)}(u,\lambda)(v_1,v_2,\ldots,v_l)&= \int_{\Omega} f^{(l)}(u,\lambda) v_1 v_2\ldots v_l \d {\bf x}.
\end{align*}
Here $f^{(l)}(t,\lambda) = \frac{\p^l}{\p t^l}f(t,\lambda)$. The equation $S^{(1)}(u,\lambda)(v) = 0$ for all $v \in E$ is a weak formulation of \eqref{eq:PDE}.
}
We consider the {bilinear form $\langle \cdot,\cdot \rangle_{E} \colon E\times E \to \R$ with}
\begin{equation}
\label{eq:NewInnerProduct}
\langle v,w \rangle_{E} =  \int_\Omega \langle \nabla v, \nabla w \rangle \d {\bf x},
\quad \qquad {v,w \in E}
\end{equation}
where $\nabla v$ and $\nabla w$ denote weak derivatives of $v,w \in E$ and $\langle \cdot,\cdot\rangle$ the scalar product in $\R^d$.
{The bilinear form $\langle \cdot,\cdot \rangle_{E}$ is symmetric, positive definite by Poincar\'e's inequality and bounded using the Cauchy-Schwarz inequality. The embedding $E = H_0^1(\Omega) \cap H^{k}(\Omega) \hookrightarrow H^{k-2}(\Omega)$ is compact \cite[Thm 6.2]{SobolevSpacesAdams}. Moreover, the Dirichlet Laplacian  $L$ is an isomorphism $H_0^1(\Omega) \cap H^{k}(\Omega) \to H^{k-2}(\Omega)$ \cite{FriedmanPDE}. Therefore, the operator $\Delta^{-1}\colon E \to E$ defined as the composition
\[
H_0^1(\Omega) \cap H^{k}(\Omega) \hookrightarrow H^{k-2}(\Omega) \stackrel{L^{-1}}{\to} H_0^1(\Omega) \cap H^{k}(\Omega)\]
is compact.
}
For each $(u,\lambda) \in E\times \R^s$ define the operator $T(u,\lambda)\colon E \to E$ by
\[
T(u,\lambda) v = -v - \Delta^{-1}(f'(u,\lambda) v).
\]
{For each $(u,\lambda) \in E \times \R^s$ the operator $T{(u,\lambda)\colon E \to E}$ is a Fredholm operator: since $v \mapsto f'(u,\lambda) v$
is continuous from $E$ into $E$ and $\Delta^{-1}\colon E \to E$ is compact, it follows 
that the composition is compact and $T(u,\lambda)\colon E \to E$ is a Fredholm operator of index 0 (Fredholm alternative). We have}
	\[
{S^{(2)}(u,\lambda)(v_1,v_2) = \langle T(u,\lambda)v_1 , v_2 \rangle_E.}
	\]
{The operator $T(u,\lambda)$ is symmetric w.r.t.\ $\langle \cdot , \cdot \rangle_E$ such that we obtain the $\langle \cdot , \cdot \rangle_E$-orthogonality of $\ker(T(u,\lambda))$ and $\rg(T(u,\lambda))$.
Since $T(u,\lambda)$ is a Fredholm operator, both spaces are closed in $E$.
Moreover, $\rg(T(u,\lambda))$ has finite codimension.
Since $T(u,\lambda)$ is symmetric, it is an elementary exercise to deduce that $E = \ker(T(u,\lambda)) \oplus \rg(T(u,\lambda))$.
The projection $\mathrm{pr}\colon E \to \rg(T(u,\lambda))$ induced by the splitting is continuous.}
For each $\lambda \in \R^s$ we define the operator $\nabla_y S(\lambda) \colon E \to \rg T(u,\lambda)$ as 
\[
u \mapsto {\mathrm{pr}(}-u - \Delta^{-1}(f(u,\lambda)){)}
\]
to $\rg T(u,\lambda)$. {We have}
\[
{S^{(1)} (u,\lambda)(v) 
= \left\langle -u - \Delta^{-1}(f(u,\lambda)),  v \right\rangle_E
= \left\langle \nabla_y S(\lambda) (u),  v \right\rangle_E}
\]
{for all $v \in \rg(T(u,\lambda))$. The observations imply the following proposition.}
\begin{prop}
The equation
\[
\forall v \in E  : \quad S^{(1)} (u,\lambda)(v)
= \int_{\Omega} \big( - \langle \nabla u, \nabla v \rangle + f(u,\lambda)v \big) \d {\bf x}
=0 
\]
is a weak formulation of \eqref{eq:PDE}. Furthermore, {assuming that $S^{(1)}(u,\lambda)=0$,} \cref{as:T1} and \cref{as:T2} hold with $\langle \cdot , \cdot \rangle_E$ \eqref{eq:NewInnerProduct}, the operator $T{(u,\lambda)\colon E \to E}$ and $\nabla_y S{(\lambda)\colon E \to \rg T(u,\lambda)}$ defined above {for each $(u,\lambda) \in E \times \R^s$}.
\end{prop}

We conclude that \cref{thm:SplittingLemmaInfinte} applies to $S$ such that the bifurcation behaviour of \eqref{eq:PDE} reduces to finite-dimensional catastrophe theory.

\subsection{Augmented systems for the example problem}\label{subsec:contaugsys}
Let us write down the augmented systems for \eqref{eq:PDE} provided by  \cref{prop:fold,prop:HilbertCusp,prop:SW,prop:butterfly}.

\paragraph{Solution}
\begin{equation}\label{eq:PDESolutionEx}
\left\{\begin{aligned}
\Delta u + f(u,\lambda) &= 0\\
u|_{\p \Omega} &=0
\end{aligned}\right.
\end{equation}

\paragraph{Fold $(A_2)$}
\begin{equation}\label{eq:FoldPDEEx}
\left\{\begin{aligned}
\Delta \alpha + f'(u,\lambda)\alpha &= 0\\
\alpha|_{\p \Omega} &=0\\
\| \alpha \|_{L^2} &=1
\end{aligned}\right.
\end{equation}

\paragraph{Cusp $(A_3)$}
\begin{equation}\label{eq:CuspPDEEx}
\int_\Omega f''(u,\lambda)\alpha^3 \d {\bf x} = 0
\end{equation}

\paragraph{Swallowtail $(A_4)$}
\begin{equation}\label{eq:SWPDEEx}
\left\{\begin{aligned}
\Delta v + f'(u,\lambda)v + f''(u,\lambda)\alpha^2 &= 0\\
v|_{\p \Omega} &=0\\
\end{aligned}\right.
\end{equation}

\[
\int_\Omega \left( f'''(u,\lambda) \alpha^4 - 3 f'(u,\lambda) v^2  \right) \ \d {\bf x}=0
\]


\paragraph{Butterfly $(A_5)$}
\begin{equation}\label{eq:ButterflyPDEEx}
\left\{\begin{aligned}
\Delta w + f'(u,\lambda)w + 3f''(u,\lambda)\alpha v + f'''(u,\lambda) \alpha^3 &= 0\\
w|_{\p \Omega} &=0\\
\end{aligned}\right.
\end{equation}
\[
\int_\Omega \left(
f''''(u,\lambda)\alpha^5
-15 f''(u,\lambda) \alpha v^2 
+10 f''(u,\lambda) \alpha^2 w 
\right) \d {\bf x}
=0.
\]


\begin{remark}
We can impose
\[\langle v, \alpha \rangle_E = \int_\Omega \langle \nabla v, \nabla \alpha \rangle \d {\bf x} =0\]
in \eqref{eq:SWPDEEx} or $\langle v, w \rangle_E = 0 $ in \eqref{eq:ButterflyPDEEx} as discussed in \cref{rem:uniquenessnotnecessary}. The uniqueness condition allows us to interpret $v$ as $F''(0)$ and $w$ as $F'''(0)$.
\end{remark}

\paragraph{Singularity $A_n$, $n \ge 4$}
The equation to determine $F^{(n-2)}(0) \colon \Omega \to \R$ reads
\begin{equation}\label{eq:ContRecAnBratu}
\left\{\begin{aligned}
\Delta F^{(n-2)}(0) + B_{n-2}(\alpha,F''(0),\ldots,F^{(n-2)}(0)) &= 0\\
\int_\Omega \langle \nabla F^{(n-2)}(0), \nabla \alpha \rangle \d {\bf x} &=0\\
F^{(n-2)}(0)|_{\p \Omega} &=0.
\end{aligned}\right.
\end{equation}
To obtain the correct expression for the term $B_{n-2}(\alpha,F''(0),\ldots,F^{(n-2)}(0))$ in \eqref{eq:ContRecAnBratu}, the Bell polynomial $B_{n-2}$ is first presented in its monomial form as in \eqref{eq:BellPolysMonForm}, then the summation sign $+$ is replaced by $+ f^{(\mathrm{degree})}(u,\lambda)$, where $\mathrm{degree}$ denotes the degree of the monomial it is multiplied with. Finally, the arguments $\alpha,F''(0),\ldots,F^{(n-2)}(0)$ are substituted into the expression.
The bifurcation test equation is given as
\begin{equation}\label{eq:ContTestAnBratu}
\int_\Omega B_n(\alpha,F''(0),\ldots,F^{(n-2)}(0),0,0) \d {\bf x} =0.
\end{equation}
(cf.\ \cref{cor:ReqJet})

To obtain the correct expression for the term $B_n(\alpha,F''(0),\ldots,F^{(n-2)}(0),0,0)$ in \eqref{eq:ContTestAnBratu} the Bell polynomial $B_{n}$ is first presented in its monomial form. Summands, which consist of a degree 2 monomial $c x_i x_j$ are replaced by
\[
c\left(-\langle \nabla F^{(i)}(0), \nabla F^{(j)}(0) \rangle + f'(u,\lambda)F^{(i)}(0)F^{(j)}(0)\right),
\]
where $c$ denotes the occurring factor. In the other summands we add $f^{(\mathrm{degree}-1)}(u,\lambda)$ as a factor, whereas $\mathrm{degree}$ is the degree of the monomial making up the summand.

\begin{remark}\label{rem:SimpleStructure}
We see that the fact that $S$ depends only quadratically on $u$ has led to a significant simplification compared to the formulas for general functionals. Indeed, the first equation in \eqref{eq:ContRecAnBratu} is of the form
\[
\Delta F^{(n-2)}(0)(x) + \kappa(x) F^{(n-2)}(0)(x) = g(x), \qquad x \in \Omega
\]
for a smooth map $g \colon \Omega \to \R$ and with $\kappa(x) = f'(u,\lambda)$. The map $ F^{(n-2)}(0)\colon \Omega \to \R$ is sought.
While the initial PDE \eqref{eq:PDE} is a semilinear Poisson equation \cite{Hsiao2006,konishi1973}, the equations which are added in the augmented system are linear PDEs.
If $\kappa (x) \ge 0$ or $\kappa (x) \le 0$ for all $x \in \Omega$ then this is the generalised Poisson equation considered in \cite{GenPoissonDipl}.
\end{remark}

\begin{example}
Let $f(u,\lambda) = \sum_{l=1}^n \frac 1 {l!} \lambda_l u^l$.
For any choice of $\lambda$ the constant function $u=0$ is a solution to the problem \eqref{eq:PDE}. The {condition for} a singularity of type $A_n$ at $(\lambda,u)=(\lambda,0)$ {is fulfilled} if and only if {$-\lambda_1$ is a simple eigenvalue of the Dirichlet Laplacian $\Delta \colon H_0^1(\Omega) \cap H^k(\Omega) \to H^{k-2}(\Omega)$}
and
\[
\lambda_i \begin{cases} =0, &   2 \le i \le n-1\\
			\not =0, &   i=n
			\end{cases}
\]
{provided that $\int_\Omega \alpha^i \d {\bf x} \not=0$ for $3 \le i \le n+1$, where $\alpha$ is the eigenfunction to the eigenvalue $-\lambda_1$.}
\end{example}

\begin{proof}
The function $u=0$ solves \eqref{eq:PDESolutionEx}. We have $f^{(l)}(0,\lambda) = \lambda_l$.
The system \eqref{eq:FoldPDEEx} reads
\[
\left\{\begin{aligned}
\Delta \alpha + \lambda_1 \alpha &= 0\\
\alpha|_{\p \Omega} &=0\\
\| \alpha \|_{L^2} &=1
\end{aligned}\right.
\]
It has a unique solution if and only if $-\lambda_1$ is a simple eigenvalue of the Dirichlet Laplacian. {The} cusp condition simplifies to $\lambda_2 =0$. Assuming \eqref{eq:PDESolutionEx}, \eqref{eq:FoldPDEEx}, \eqref{eq:CuspPDEEx} we see that $F''(0)=0$ solves \eqref{eq:ContRecAnBratu} with $n=4$. The swallowtail condition is fulfilled if and only if $\lambda_3=0$. Assuming that $F''(0)=\ldots=F^{(t-2)}(0) = 0$ and $\lambda_3=\lambda_4=\ldots=\lambda_{t-1}=0$ we see that $F^{(t-1)}(0)=0$ solves \eqref{eq:ContRecAnBratu} with $n=t+1$. The bifurcation test equation \eqref{eq:ContTestAnBratu} for $A_{t+1}$ is fulfilled if and only if $\lambda_{t}=0$. The claim follows by induction.
\end{proof}


\subsection{Numerical experiment}

{The classical Bratu problem considers the PDE
\[
\Delta u + \lambda_1 \exp\left( \frac{u}{1+\lambda_2 u} \right) = 0
\]
on a $d$-dimensional cube with zero Dirichlet boundary values. The boundary value problem is popular to study fold and cusp bifurcations \cite{Mohsen201426}.}
In the following we consider the domain $\Omega = (0,1)\times(0,1)$
and set
\begin{align*}
f(t,\lambda)=\lambda_1 \exp\left({\frac t{\lambda_2 t+1}}\right)+\lambda_3 \sin(\lambda_1 t)
\end{align*}
{in the Dirichlet problem \eqref{eq:PDE} given as
\begin{equation*}
\left\{\begin{aligned}
\Delta u + f(u,\lambda) &= 0\\
u|_{\p \Omega} &=0.
\end{aligned}\right.
\end{equation*}
The considered problem coincides with the Bratu problem for $\lambda_3 = 0$.}
Numerical experiments have been performed in \cite{Continuation}.
{The third parameter has been added to create a swallowtail bifurcation, which we will find numerically.
As the topological boundary $\p \Omega$ of the domain is not regular, the assumptions in \cref{subsec:Setup} do not hold. However, the following numerical experiment illustrates on a classical example how the derived augmented systems can be used to locate bifurcations.}
As before, the primitive w.r.t.\ the first argument is denoted by $\bar f$, i.e.\ $\frac{\p}{\p t} \bar f (t,\lambda) = f(t,\lambda)$.

\subsubsection{Discretisation via a discrete Lagrangian method}

In view of \cref{sec:SplittingandBifurTest,subsec:Setup} it appears natural to use a variational method (see \cite[VI.6]{GeomIntegration} or \cite{MarsdenWestVariationalIntegrators}, for instance) to discretise \eqref{eq:PDE}.
We obtain a mesh on $\Omega$ as the cross product of a uniform mesh in the $x$ direction with $N$ interior mesh points and spacing $\Delta x = \frac 1 {N+1}$ and a uniform mesh in the $y$ direction with $M$ interior mesh points and spacing  $\Delta y = \frac 1 {M+1}$.
Real-valued, continuous functions $u \colon \Omega \to \R$ are represented as matrices $U \in \R^{N \times M}$ where the component $U_{i,j}$ corresponds to the value $u(i \Delta x, j \Delta y)$ on the interior grid. Alternatively, we can flatten the matrix $U$ to a vector $\bar u \in \R^{N\cdot M}$ such that $U_{i,j}=\bar u_{(j-1)N+i}$.
For the functional 
\[
S(u,\lambda) = \int_{\Omega} \left( - \frac 12  (u_x^2+u_y^2)  + \bar f(u,\lambda)\right) \d x \d y
\]
from \cref{subsec:Setup} we consider the discrete functional
\[
S_\Delta(U,\lambda) 
= \sum_{i=0}^N \sum_{j=0}^M \left( 
- \frac 12 \left( \frac{ (U_{i,j}-U_{i+1,j})^2}{\Delta x^2} + \frac{ (U_{i,j}-U_{i,j+1})^2}{\Delta y^2} \right) + \bar f(U_{i,j},\lambda)
\right)
\]
with $U_{0,j}=U_{i,0}=U_{N+1,j}=U_{i,M+1}=0$.
For $k \in \N$ define 
\[
\overline D (k) := \begin{pmatrix}
-2&1\\
1&-2&1\\
&\ddots & \ddots & \ddots\\
&&1&-2&1\\
&&&1&-2\\
\end{pmatrix} \in \R^{k \times k}.
\]
The matrices
\[
D_{xx} = \frac 1 {\Delta x^2} \overline D(N) \in \R^{N \times N}, \quad
D_{yy} = \frac 1 {\Delta y^2} \overline D(M) \in \R^{M \times M}
\]
are the standard central finite-difference discretisations of the operators $\frac{\p^2}{\p x^2}$ and $\frac{\p^2}{\p y^2}$.
Define
\[
L:=\Id_M \otimes D_{xx} + D_{yy}\otimes \Id_N
\]
where $\otimes$ denotes the Kronecker tensor product. Notice that 
\[
S_{\Delta}(\bar u,\lambda) = \frac 12 \bar u^T L \bar u + \bar f(\bar u,\lambda).
\]
In the expression above the function $\bar f$ is evaluated component-wise, i.e\ the $k^{\text{th}}$ component of $\bar f(u,\lambda)$ is given by $\bar f(u_k,\lambda)$.
We have
\begin{equation}\label{eq:DSDelta}
\D_u S_\Delta (\bar u,\lambda) = L \bar u + f(\bar u,\lambda). 
\end{equation}

\begin{remark}\label{rem:DiscreteLagrangianMethod}
The expression \eqref{eq:DSDelta} is the flattening of the matrix-valued function
\[
(U,\lambda) \mapsto U D_{yy} + D_{xx} U + f(U,\lambda),
\]
which can be obtained by a direct discretisation of the PDE \eqref{eq:PDE} using standard central difference approximations for second derivatives.
We see that the equation $\D_u S_\Delta (\bar u,\lambda) = 0$ obtained from an application of the discrete Lagrangian principal to $S$ leads to the same equations as a discretisation of the Laplacian operator using central finite-differences.
\end{remark}

\subsubsection{Application of pseudoarclength continuation to augmented discrete systems}\label{subsec:ContiAugDiscrete}

To locate a singularity of a given type one can write down the augmented system of the singularity and solve the system using an iterative solver. This is called a {\em direct method}.
For the convergence of the solver a good initial guess is required which is usually not available a priori. Instead, one may employ a continuation method and first continue a known solution along one parameter until a fold singularity is detected, then augment the system using the fold bifurcation test equation and continue a line of folds until a cusp is detected and so on. In the following numerical experiment we will use pseudoarclength continuation \cite{DoedelConti}. Our strategy of performing conceptually simple one-dimensional continuations of singularities of high codimension can be contrasted to higher-dimensional continuation of solutions \cite{Henderson2007}. More information on continuation methods can be found in
\cite{Allgower_2003,Continuation,Deuflhard2011,KrauskopfNumericalContinuation}.

Whenever a bifurcation is detected one has the option to locate the singularity of the discrete system exactly using a direct method.
Generally speaking, unless one has prior knowledge about the bifurcation diagram or is interested only in a specific parameter region one needs to search in different directions for bifurcations. 
Once arrived at a high codimensional bifurcation, the discretisation parameter is decreased and a direct method is applied to approximate the location of the singularity in the continuous system.
As starting values interpolated data from the coarser systems can be utilised.

To simplify notation in the following we will write $u$ instead of $\bar u$, $G$ for $\D S_\Delta$ and $G_u$ for the Jacobian matrix $\D_u G$.
%
%
%
\paragraph{Fold}
Fixing $(\lambda_2,\lambda_3)=(0,0)$ and starting at $\lambda_1=0$ we continue the solution $u=0\in \R^{N \cdot M}$ by applying pseudoarclength continuation to \[G(u(s),\lambda_1(s),\lambda_2,\lambda_3)=0\] until we find a fold. The singularity is detected when $\lambda_1(s)$ changes from being increasing to decreasing in $s$. (See the left plot in \cref{fig:Conti}.)

\begin{figure}
\begin{center}
\includegraphics[width=0.45\textwidth]{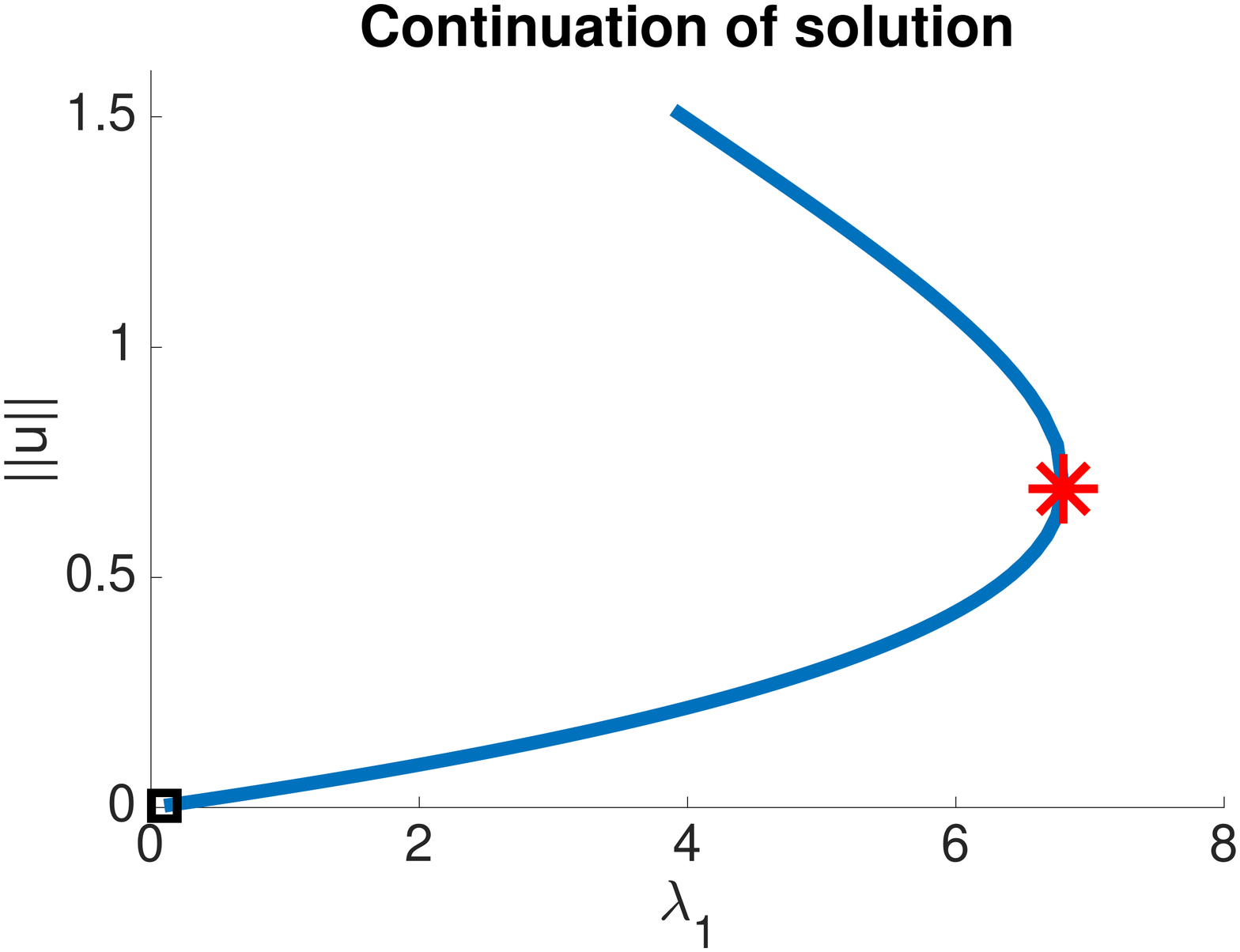}
\includegraphics[width=0.45\textwidth]{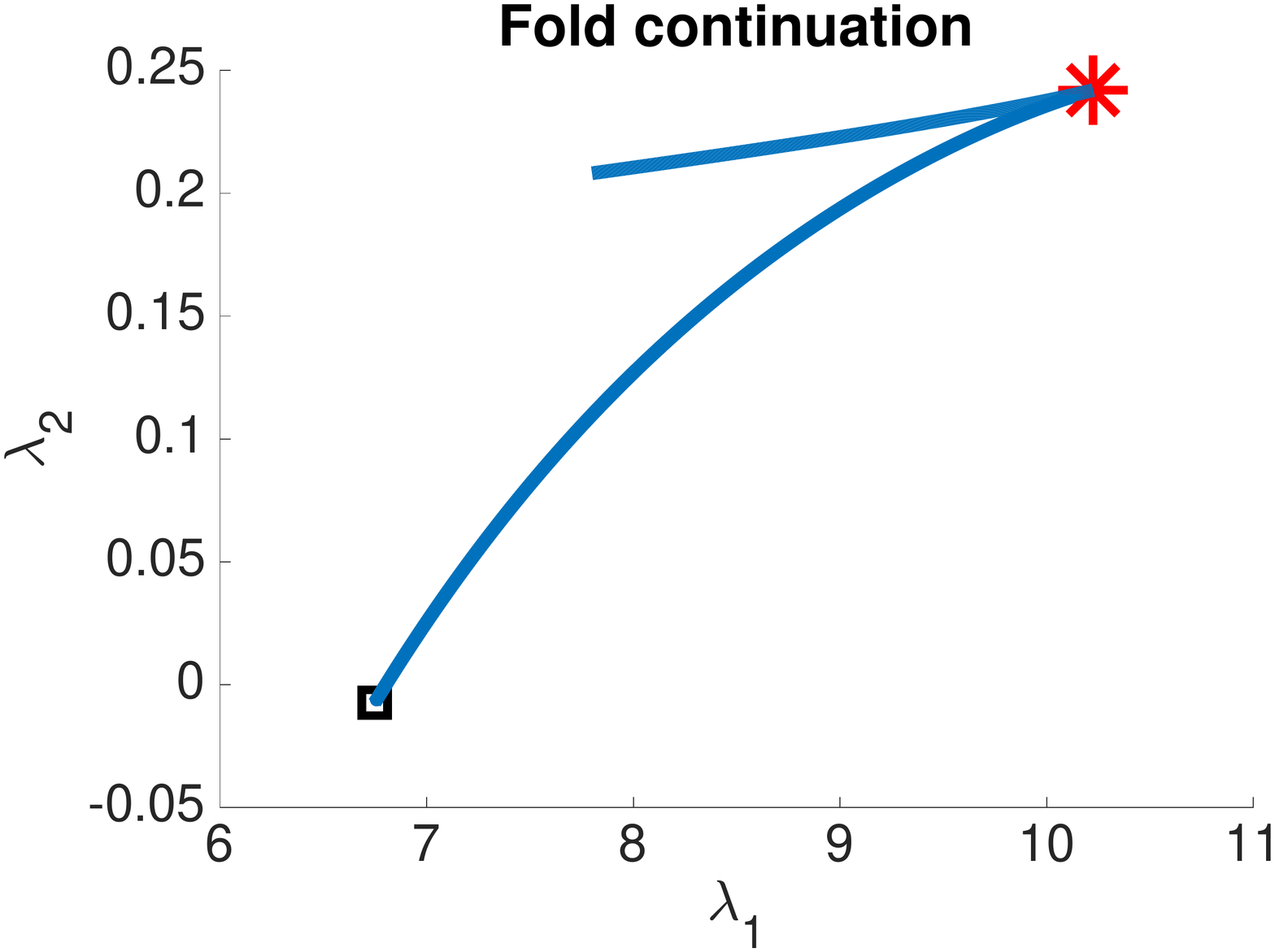}

\end{center}
\caption{Numerical experiment with $(N,M)=(15,15)$.
The left figure shows a continuation of the solution $(u,\lambda)=(0,0)$ of $G(u,\lambda_1,\lambda_2,\lambda_3)=0$ along $\lambda_1$ while keeping $\lambda_2=0=\lambda_3$ fixed. A fold point is marked by $\ast$.
The right figure shows a continuation of the fold singularity by varying $\lambda_1$ and $\lambda_2$ while keeping $\lambda_3=0$ fixed. A cusp point is marked by $\ast$.
In each plot the square marker shows where the continuation was initiated.
}\label{fig:Conti}
\end{figure}

\paragraph{Cusp}
We allow $\lambda_2$ to be parameter dependent as well and apply pseudoarclength continuation to
\[
F^{(1)}(u(s),\alpha(s),\lambda_1(s),\lambda_2(s))=
\begin{pmatrix}
G(u(s),\lambda_1(s),\lambda_2(s),\lambda_3)\\
G_u(u(s),\lambda_1(s),\lambda_2(s),\lambda_3) \alpha(s)\\
\alpha(s)^T \alpha(s) - \Delta x \Delta y
\end{pmatrix}=0
\]
starting with the approximated fold data and a random, normalised guess for $\alpha$. (See the right plot in \cref{fig:Conti}.) During the continuation process we monitor the cusp condition
\begin{equation}\label{eq:CuspNumEx}
f_{uu}(u,\lambda_1(s),\lambda_2(s),\lambda_3)^T \alpha^{3} = 0,
\end{equation}
where raising $\alpha$ to the third power is to be understood component-wise. After detecting a change of sign in the left-hand side of \eqref{eq:CuspNumEx}, we improve the accuracy of the location of the singular point by solving
\[
\begin{pmatrix}
G(u(s),\lambda_1(s),\lambda_2(s),\lambda_3)\\
G_u(u(s),\lambda_1(s),\lambda_2(s),\lambda_3) \alpha(s)\\
\alpha(s)^T \alpha(s) - \Delta x \Delta y\\
f_{uu}(u,\lambda_1(s),\lambda_2(s),\lambda_3)^T \alpha^{3}
\end{pmatrix}=0
\]
using Newton iterations. 

\paragraph{Swallowtail}
We allow an $s$-dependence of $\lambda_3$ and apply pseudoarclength continuation to the system
\[
F^{(2)}(u(s),\alpha(s),\lambda_1(s),\lambda_2(s),\lambda_3(s))=
\begin{pmatrix}
G(u(s),\lambda_1(s),\lambda_2(s),\lambda_3(s))\\
G_u(u(s),\lambda_1(s),\lambda_2(s),\lambda_3(s)) \alpha(s)\\
\alpha(s)^T \alpha(s) - \Delta x \Delta y\\
f_{uu}(u,\lambda_1(s),\lambda_2(s),\lambda_3(s))^T \alpha^{3}
\end{pmatrix}=0
\]
starting from the calculated cusp position. During this process, we monitor the swallowtail condition 
\begin{equation}\label{eq:SWNumEx}
f_{uuu}(u,\lambda_1,\lambda_2,\lambda_3)^T \alpha^{4} +6 (f_{uu} (u,\lambda) . \alpha^{2})^T v + 3 v^T G_u(u,\lambda_1,\lambda_2,\lambda_3) v = 0,
\end{equation}
where $f_{uu} (u,\lambda) . \alpha^{2}$ denotes the component-wise product of the vectors $f_{uu} (u,\lambda)$ and $\alpha^{2}$. In each continuation step the vector $v$ is obtained as follows: we calculate the vector $(v^T,t)^T\in \R^{N \cdot M+1}$ which minimises the euclidean norm of
\[
\begin{pmatrix}
G_u (u,\lambda_1,\lambda_2,\lambda_3) & \alpha
\end{pmatrix}
\begin{pmatrix}
v\\ t
\end{pmatrix}
- f_{uu} (u,\lambda) . \alpha^{2}.
\]
For this we calculate the (under non-degeneracy conditions unique) solution to
\[
\left(G_u (u,\lambda_1,\lambda_2,\lambda_3)\cdot G_u (u,\lambda_1,\lambda_2,\lambda_3)^T + \alpha \alpha^T \right) \bar v = -f_{uu} (u,\lambda)
\]
and obtain $v$ as $v = G_u (u,\lambda_1,\lambda_2,\lambda_3)^T \bar v$. (Notice that transposition can be neglected since $G_u(u,\lambda_1,\lambda_2,\lambda_3)$ is symmetric.)

As the left-hand side of \eqref{eq:SWNumEx} changes sign, we detect a candidate for a swallowtail point at $(u_{\mathrm{SW}},\lambda^{\mathrm{SW}}) = (u_{\mathrm{SW}},\lambda^{\mathrm{SW}}_1,\lambda^{\mathrm{SW}}_2,\lambda^{\mathrm{SW}}_3)$.
(See \cref{fig:SWCondMonitor}.)
Indeed, the swallowtail condition \eqref{eq:SWNumEx} has a regular root at the swallowtail point which implies that the singularity is not further degenerate.

In \cref{fig:SWLineofCuspsLineofFolds} we do a fold continuation using pseudoarclength continuation applied to $F^{(1)}$ starting at a cusp point near $(u_{\mathrm{SW}},\lambda^{\mathrm{SW}})$ while fixing $\lambda_3$. The fold line is continued in both directions and shows the characteristics of a swallowtail bifurcation. 
This verifies that $(u_{\mathrm{SW}},\lambda^{\mathrm{SW}})$ is indeed a swallowtail point of the discretised system.

\begin{remark}
As is known from catastrophe theory \cite{Arnold1,lu1976singularity}, the only generic\footnote{persistent under small perturbations. For exact notions see \cite{Arnold1}, for instance.} bifurcations of codimension smaller or equal to 3 which critical points of a function $R \colon \R^n \to \R$ can undergo are $A_2$, $A_3$, $A_4$, $D^+_4$ and $D^-_4$. However, if a singularity $D^+_4$ or $D^-_4$ occurs along a line of cusp bifurcations then the swallowtail condition tends to $+\infty$ or $-\infty$ as one approaches the singularity. (See \cref{fig:SWCondMonitorD+,fig:SWCondMonitorD-}.) Moreover, at the singularity, $\mathrm {Hess}(R)(z)$ has a 2-dimensional kernel and the value of the swallowtail condition is not defined.
\end{remark}


\begin{figure}
\begin{center}
\includegraphics[width=0.45\textwidth]{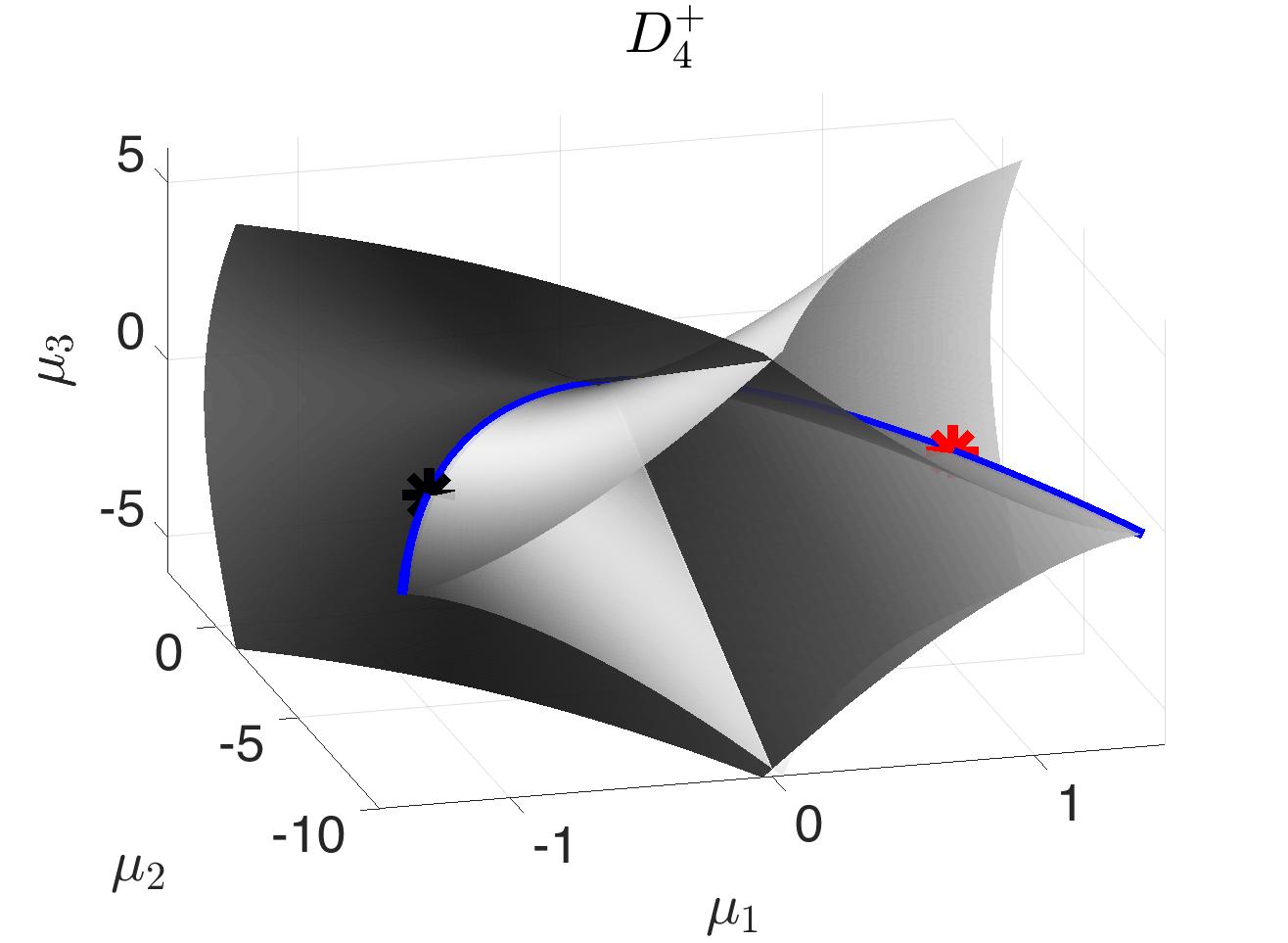}
\quad
\includegraphics[width=0.45\textwidth]{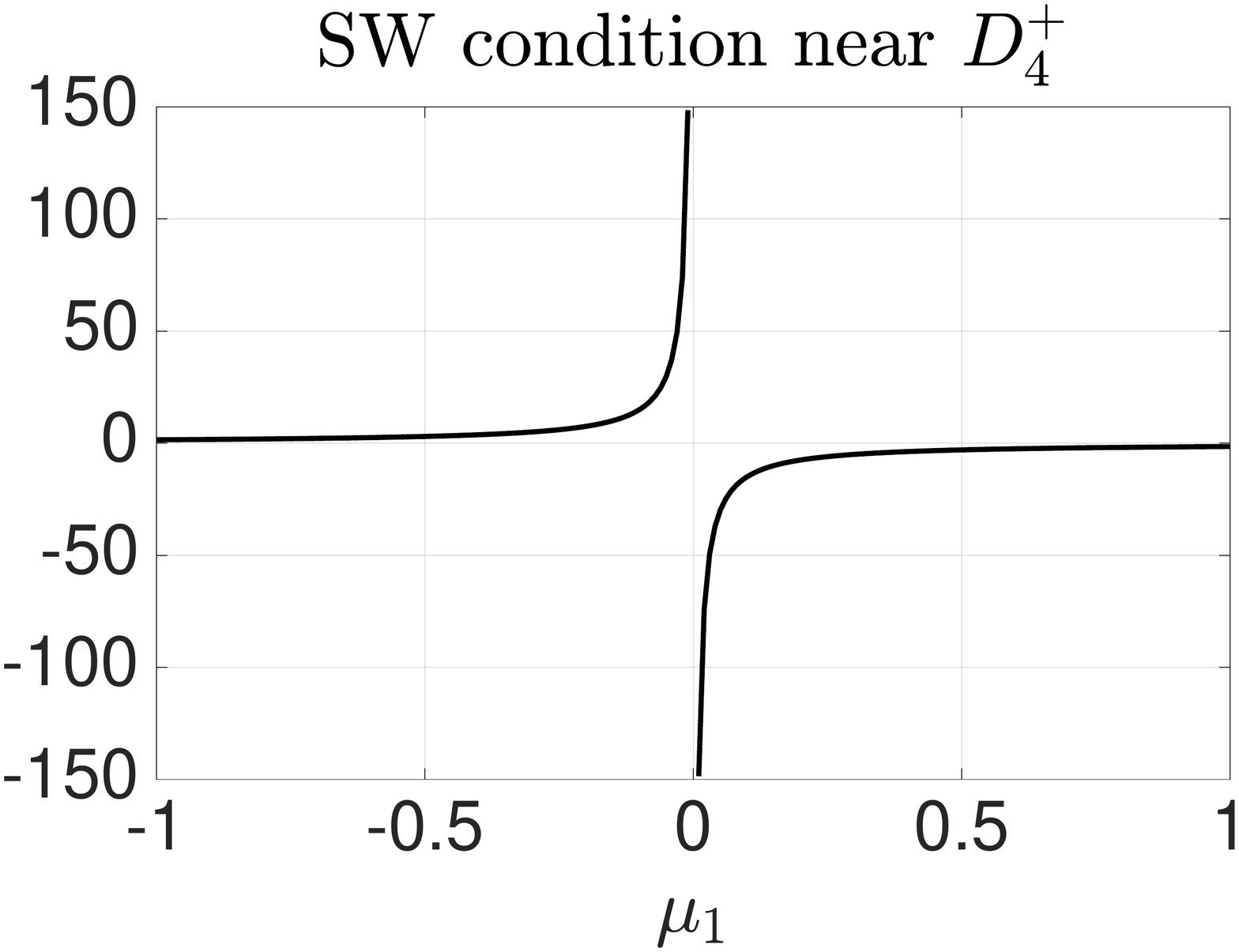}
\end{center}
\caption{
The left figure shows a singularity of type $D_4^+$, called {\em hyperbolic umbilic singularity}. It is a generic codimension 3 singularity in catastrophe theory. Here $\mu=(\mu_1,\mu_2,\mu_3) $ are parameters. Each point on the blue line (the edge) corresponds to a cusp singularity with an exception at $\mu=(0,0,0)$ where the $D_4^+$ singularity is located.
The two points marked with asterisks are cusp points with different signatures. 
The figure to the right shows the value of the bifurcation test equation for $A_4$ along the line of cusps parametrised by $\mu_1$. We see that the bifurcation test equation has a singularity at $\mu_1=0$, where the hyperbolic umbilic singularity is located, and that it changes sign at the singularity.
}\label{fig:SWCondMonitorD+}
\end{figure}

\begin{figure}
\begin{center}
\includegraphics[width=0.45\textwidth]{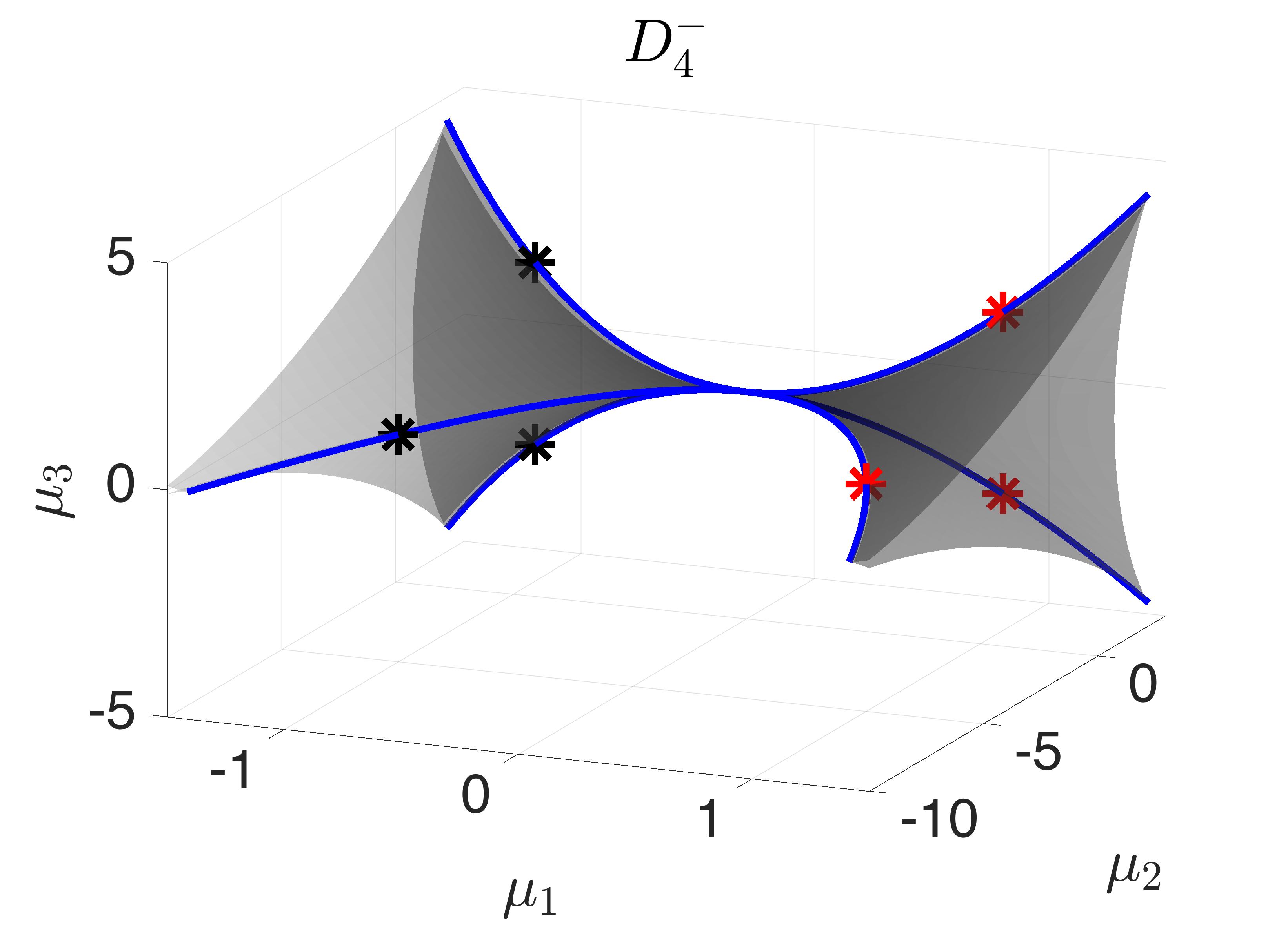}
\quad
\includegraphics[width=0.45\textwidth]{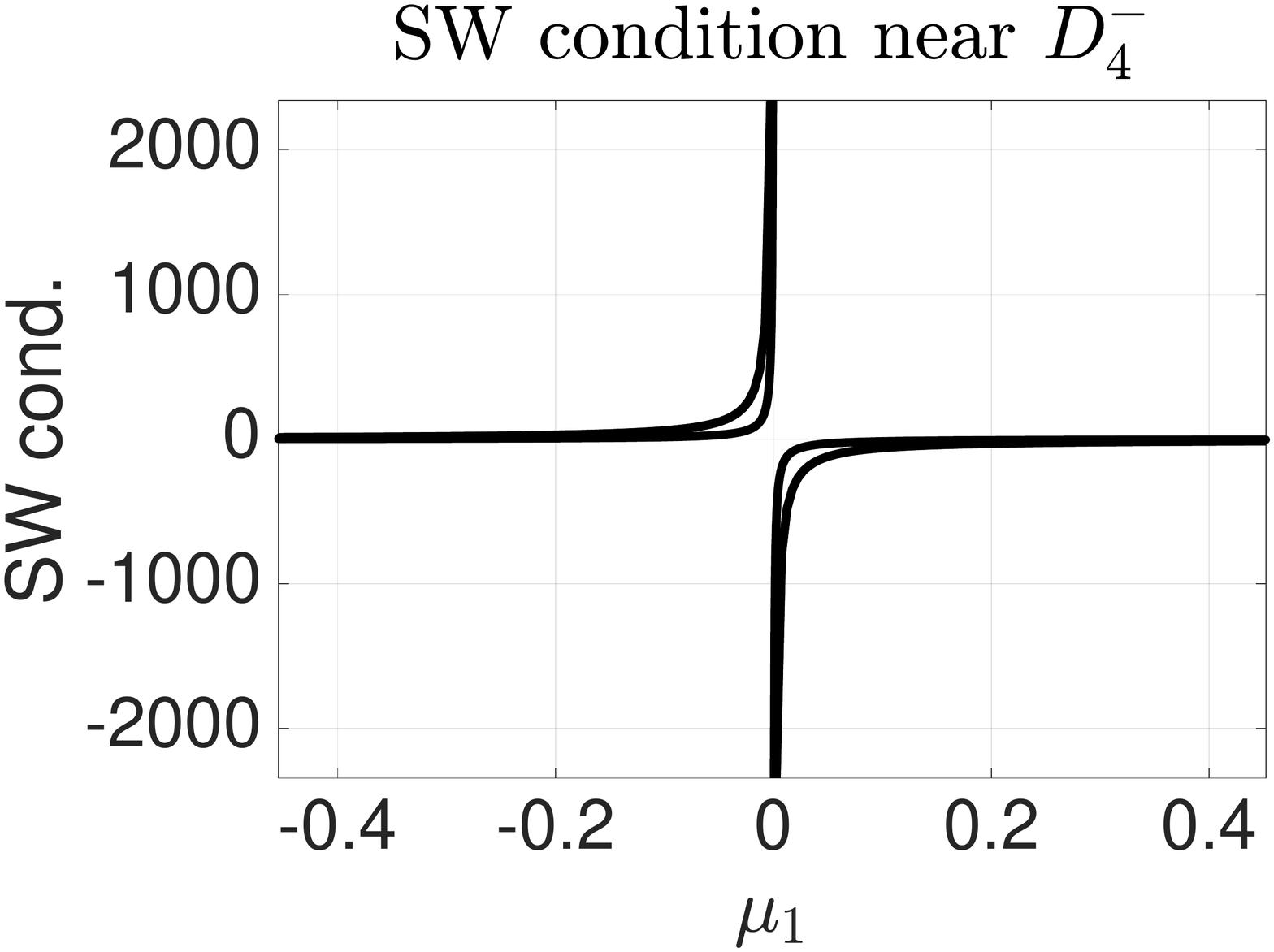}
\end{center}
\caption{
The left figure shows a singularity of type $D_4^-$, called {\em elliptic umbilic}. It is a generic codimension 3 singularity in catastrophe theory. Here $\mu=(\mu_1,\mu_2,\mu_3) $ are parameters. Each point on the blue lines (the edges) corresponds to a cusp singularity with an exception at $\mu=(0,0,0)$ where the $D_4^-$ singularity is located.
The three points on the line of cusps marked with asterisks with negative $\mu_1$ values are of the positive type. The other three marked points are of the negative type.
The figure to the right shows the value of the bifurcation test equation for $A_4$ (swallowtail condition) along the line of cusps parametrised by $\mu_1$. Due to symmetries in this example, two lines are plotted above each other in the $\mu_1<0$ regime as well as in the $\mu_1>0$ regime. When the elliptic singularity is approached along a line of negative cusps then the swallowtail condition tends to $-\infty$. Analogously, it tends to $+\infty$ if the singularity is approached through positive cusp points.
}\label{fig:SWCondMonitorD-}
\end{figure}

\begin{figure}
\begin{center}
\includegraphics[width=0.45\textwidth]{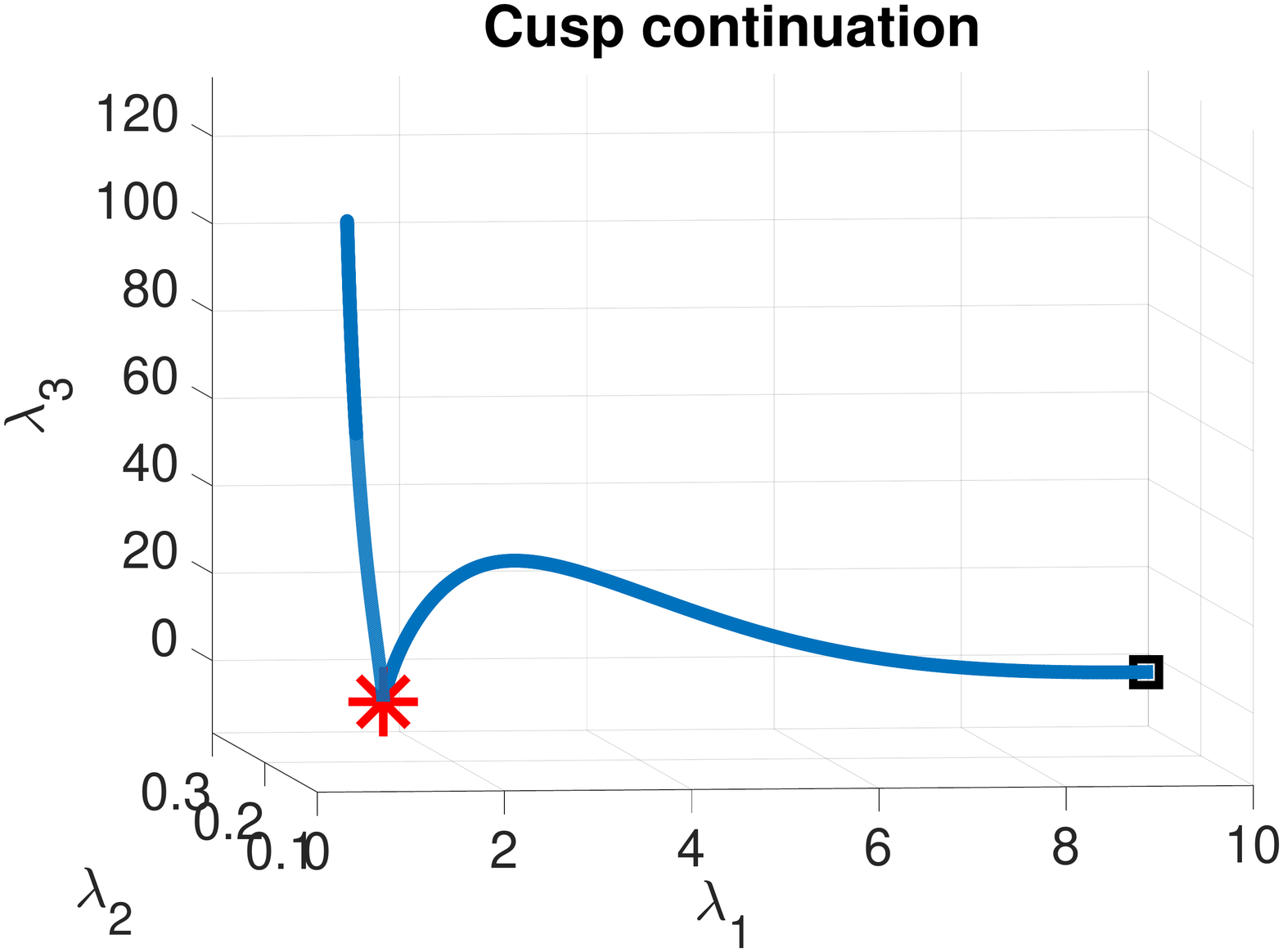}
\quad
\includegraphics[width=0.45\textwidth]{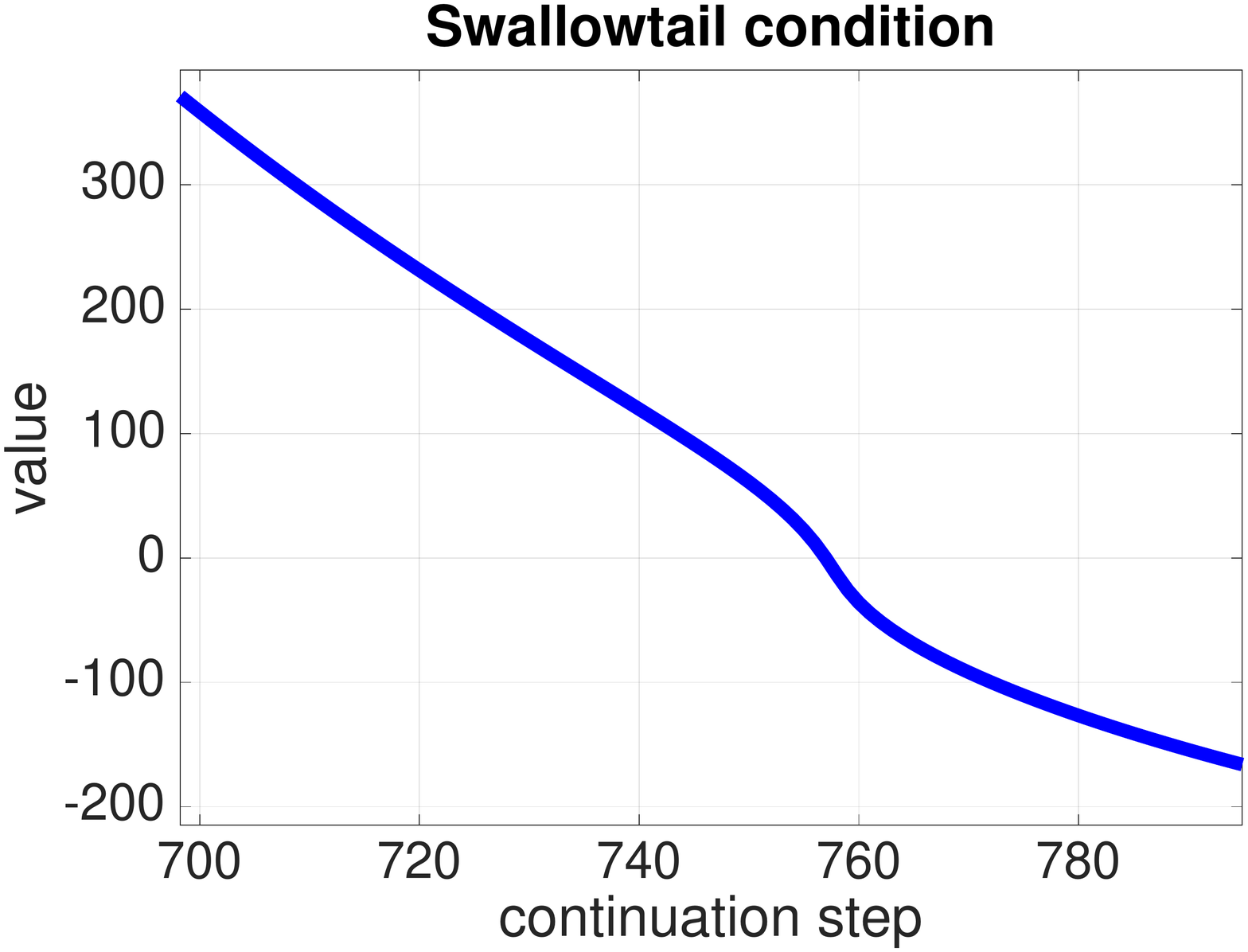}
\end{center}
\caption{
The left figure shows a continuation of the cusp singularity detected in \cref{fig:Conti}
by varying $\lambda_1$, $\lambda_2$ and $\lambda_3$. A swallowtail point is marked by $\ast$. The square marker shows where the continuation was initiated. As the line of cusps is numerically continued the swallowtail condition is monitored.
To the right we see the value of the swallowtail condition versus the number of continuation steps. The root corresponds to the swallowtail point. Since the root is non-degenerate, we know that the butterfly condition is not fulfilled such that the singularity is indeed a swallowtail point and not further degenerate.
}\label{fig:SWCondMonitor}
\end{figure}

\begin{figure}
\begin{center}
\includegraphics[width=0.5\textwidth]{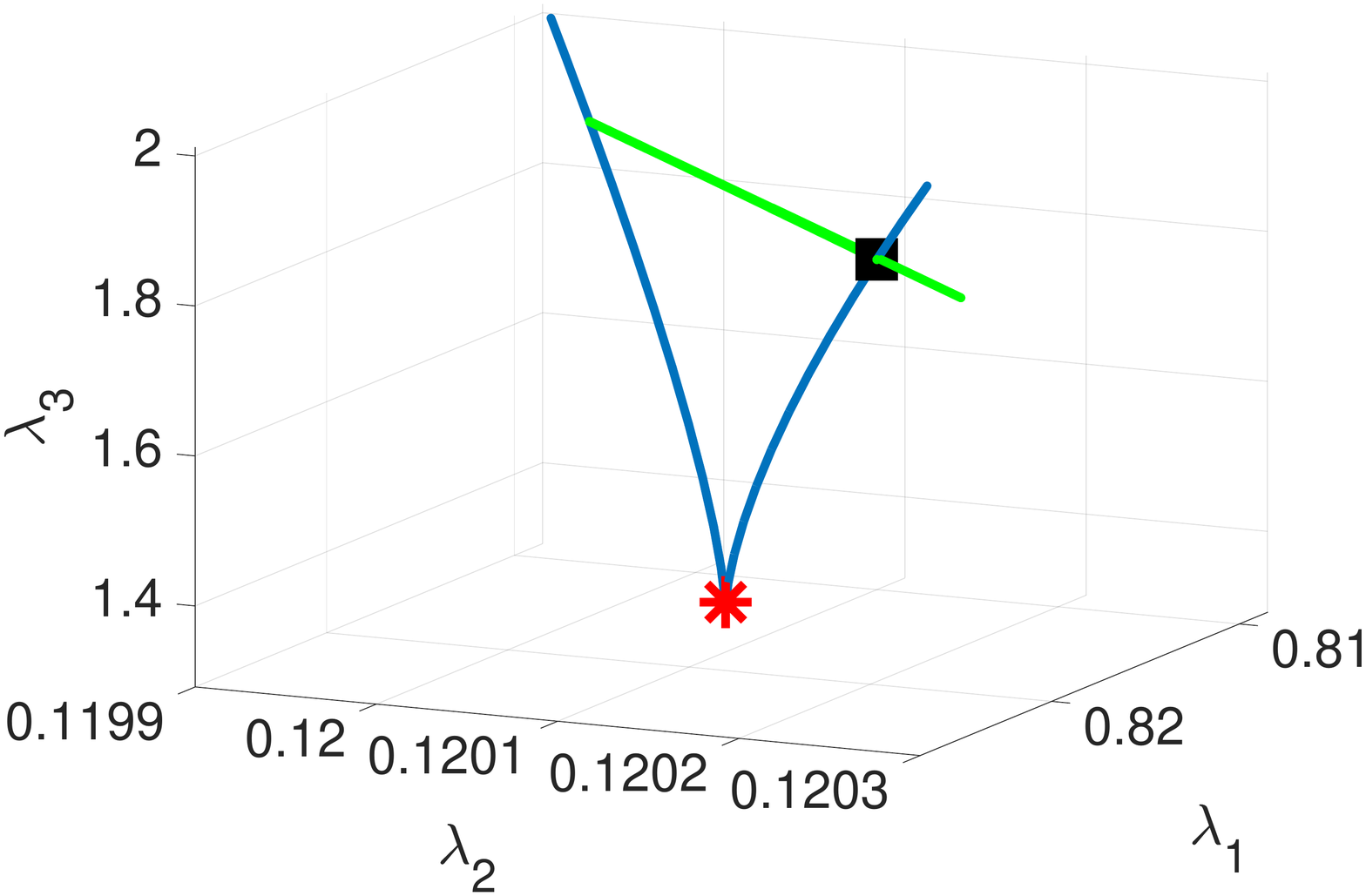}
\quad
\includegraphics[width=0.45\textwidth]{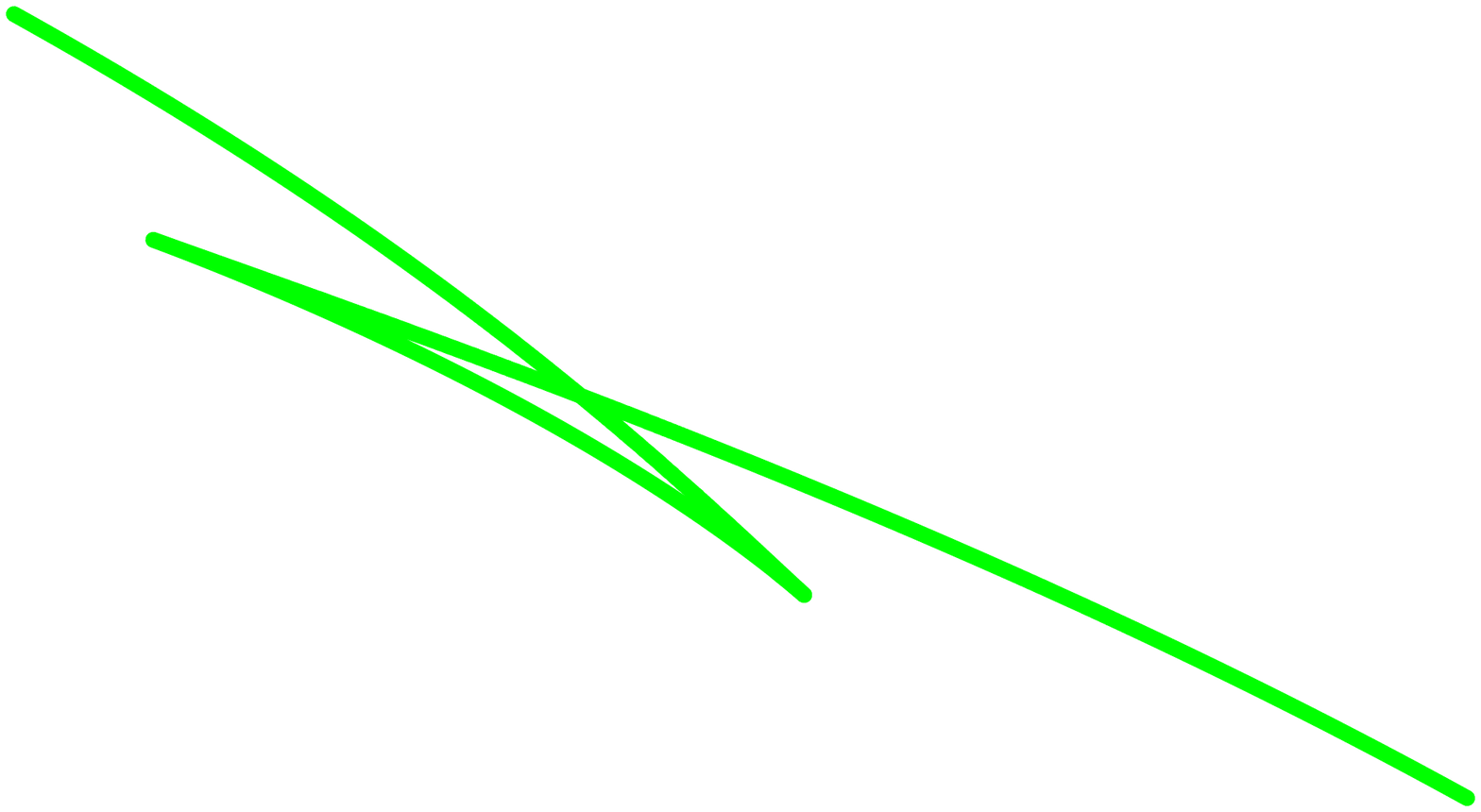}
\end{center}
\caption{Numerical experiment with $(N,M)=(10,10)$ showing a line of cusp points in the figure to the left. At the point marked with an asterisk the swallowtail condition \eqref{eq:SWNumEx} is fulfilled. From the point marked with a square we continue a branch of folds fixing the parameter $\lambda_3$. The figure to the right shows a projection of the line of folds along the $\lambda_3$-direction to the $\lambda_1$, $\lambda_2$ plane after rotation and rescaling. This is done to to make the structure of the branch visible as it is very flat in the original coordinates $(\lambda_1,\lambda_2,\lambda_3)$.
The plot is characteristic for swallowtail bifurcations (compare with \cref{fig:AseriesIllus}).
}\label{fig:SWLineofCuspsLineofFolds}
\end{figure}

\begin{remark}
Forming augmented systems for the discrete functional $S_\Delta$ (recall $\D_u S_\Delta = G$) gives rise to the same system of equations as discretising the continuous augmented systems from \cref{subsec:contaugsys}. In other words, forming the augmented system commutes with discretisation.
This extends \cref{rem:DiscreteLagrangianMethod}.
\end{remark}

\subsubsection{Determination of the position and convergence of the swallowtail point }

To calculate the position of the bifurcation point more accurately we apply Newton iterations to $F^{(3)}(u,\alpha,\bar v,\lambda)$ given as
\[F^{(3)}(u,\alpha,\bar v,\lambda)=
\begin{pmatrix}
G(u,\lambda)\\
G_u(u,\lambda) \alpha\\
\Delta x \Delta y\alpha^T \alpha -1\\
f_{uu}(u,\lambda)^T \alpha^{3}\\
\left(G^2_u (u,\lambda) + \alpha \alpha^T \right) \bar v +f_{uu} (u,\lambda)\\
f_{uuu}(u,\lambda)^T \alpha^{4} +6 (f_{uu} (u,\lambda) . \alpha^{2})^T G_u (u,\lambda) \bar v + 3 \bar v G^3_u (u,\lambda)\bar v
\end{pmatrix}
\]
until convergence using $(u_{\mathrm{SW}},\alpha_{\mathrm{SW}}, 0_{NM \times 1}, \lambda^{\mathrm{SW}})$ as initial guess. We obtain a root $(u^1_{\mathrm{SW}},\alpha^1_{\mathrm{SW}}, v^1_{\mathrm{SW}}, \lambda^1_{\mathrm{SW}})$ with $(u^1_{\mathrm{SW}},\alpha^1_{\mathrm{SW}},\lambda^1_{\mathrm{SW}})$ near $(u_{\mathrm{SW}},\alpha_{\mathrm{SW}},\lambda^{\mathrm{SW}})$.
We successively increase $(N,M)=(N,N)$ and repeat the process of finding a root of $F^{(3)}$. As initial guesses for Newton's method in the $j^{\mathrm{th}}$-step we use $\lambda^j_{\mathrm{SW}}$ from the previous calculation and linear interpolations of $u^j_{\mathrm{SW}}$, $\alpha^j_{\mathrm{SW}}$ and $v^j_{\mathrm{SW}}$ on the new grid with Dirichlet boundary conditions.
The Jacobian matrix of $F^{(3)}$ required for the Newton iterations is given as
\[
\D F^{(3)} = 
\begin{pmatrix}
G_u &0_{NM \times NM}&0_{NM \times NM}& \D_{\lambda} f\\
\diag(f_{uu}.\alpha) & G_u & 0_{NM \times NM} & \D_{\lambda} f_u . \alpha \\
0_{1 \times NM}& 2\Delta x \Delta y \alpha^T &0_{1\times NM}&0_{1 \times 3} \\
(f_{uuu}.\alpha^3)^T & 3 (f_{uu}.\alpha^2)^T & 0_{1 \times NM} & (\D_{\lambda} f_{uu}^T \alpha^3)^T\\
\D_u F^{(3)}_5
& \alpha \bar v^T+\alpha^T\bar v \Id_{NM}
& G_u^2 + \alpha \alpha^T
& \D_\lambda F^{(3)}_5\\
\D_uF_6^{(3)}
& \D_\alpha F_6^{(3)}
& \D_v F_6^{(3)}
& \D_\lambda F_6^{(3)}
\end{pmatrix}
\]
with
\begin{align*}
L&=\Id_M \otimes D_{xx} + D_{yy}\otimes \Id_N\\
\D_u F^{(3)}_5
&=L .(f_{uu}.\bar v )+\diag(f_{uu}. L \bar v+2f_{u}.f_{uu}.\bar v+f_{uuu})\\
\D_\lambda F^{(3)}_5
&=
L( (\D_\lambda f_u) .\bar v)+(\D_\lambda f_u).(L\bar v)+2(\D_\lambda f_u).f_u.\bar v+\D_\lambda f_{uu}\\
DQ^{(u)}&=2(f_{uu}.\bar v)^T.(\bar v^TL^2)
        +(f_{uu}.(L\bar v))^T.(\bar v^TL)
        +4(f_u.f_{uu}.\bar v)^T.(\bar v^TL)\\
        &+(f_{uu}.\bar v)^T.((f_u.\bar v)^TL)
        +((f_{uu}.\bar v).(L(f_u.\bar v)))^T
        +3(f_u^2.f_{uu}.\bar v)^T.\bar v^T\\
\D_uF_6^{(3)}
&=
(f_{uuuu}.\alpha ^4)^T 
       +6((f_{uuu}.\alpha ^2).(G_u\bar v))^T 
       +6(f_{uu}^2.\bar v.\alpha ^2)^T 
     +3DQ^{(u)}\\
\D_\alpha F_6^{(3)} &=
4(f_{uuu}.\alpha ^3)^T+12(f_{uu}.\alpha .(G_u\bar v))^T\\
\D_{\bar v} F_6^{(3)} &=
6(f_{uu}.\alpha ^2)^TG_u + 6(G^3_u \bar v)^T\\
DQ^{(\lambda)}&= 
2\bar v^TL^2((\D_\lambda f_u).\bar v)
+\bar v^TL((\D_\lambda f_u).(L\bar v))
+4\bar v^TL((\D_\lambda f_u).f_u.\bar v)\\
&+2(f_u.\bar v)^TL((\D_\lambda f_u).\bar v)+3\bar v^T((\D_\lambda f_u).f_u^2.\bar v)\\
\D_\lambda F_6^{(3)} &= ((\D_\lambda f_{uuu})^T\alpha ^4)^T 
                        +6(((\D_\lambda f_{uu}).\alpha ^2)^TG_u\bar v)^T \\
                        &+6(f_{uu}.\alpha ^2)^T((\D_\lambda f_u).\bar v) 
                        +3DQ^{(\lambda)}
\end{align*}
Here we use the convention that ``.'' denotes point-wise multiplication. Point-wise multiplication of a column vector with a matrix means that the vector is multiplied point-wise with each column of the matrix. Analogously for row vectors. Moreover, the power of a vector is to be understood point-wise.
Zero matrices of dimension $s\times t$ are denoted by $0_{s\times t}$ and identity matrices by $\Id_{s \times t}$.

\begin{remark}
Appropriate sub-matrices of $\D F^{(3)}$ correspond to $\D F^{(1)}$ and $\D F^{(2)}$. These have been used for the pseudoarclength continuation described in \cref{subsec:ContiAugDiscrete}. The Jacobian matrices $\D F^{(1)}$, $\D F^{(2)}$ and $\D F^{(3)}$ and the matrices $L$, $G_u$ involved in the function evaluations $F^{(1)}$, $F^{(2)}$ and $F^{(3)}$ are sparse and represented using an appropriate datatype in the numerical calculations.
Moreover, compared to equations for a general functional $S$, the structure of the semilinear Poisson equation dramatically simplifies occurring equations and numerical complexity because the multilinear operators $G_{uu}$, $G_{uuu}$, \ldots are diagonal and their diagonal is given by an evaluation of an appropriate derivative of $f$ at $u$. The simplicity of the structure has been observed in the continuous setting in \cref{rem:SimpleStructure}.
\end{remark}

\Cref{fig:SWConvergence} shows that as $N=M$ increases, the position of the swallowtail point converges in the parameter space. We also observe that the data $(u,\alpha,v)$ converges to a fixed shape. \Cref{fig:SWSolution} shows the shape of $(u,\alpha,v)$ at the approximated swallowtail point for $(N,M)=(85,85)$.

\begin{figure}

\includegraphics[width=0.3\textwidth]{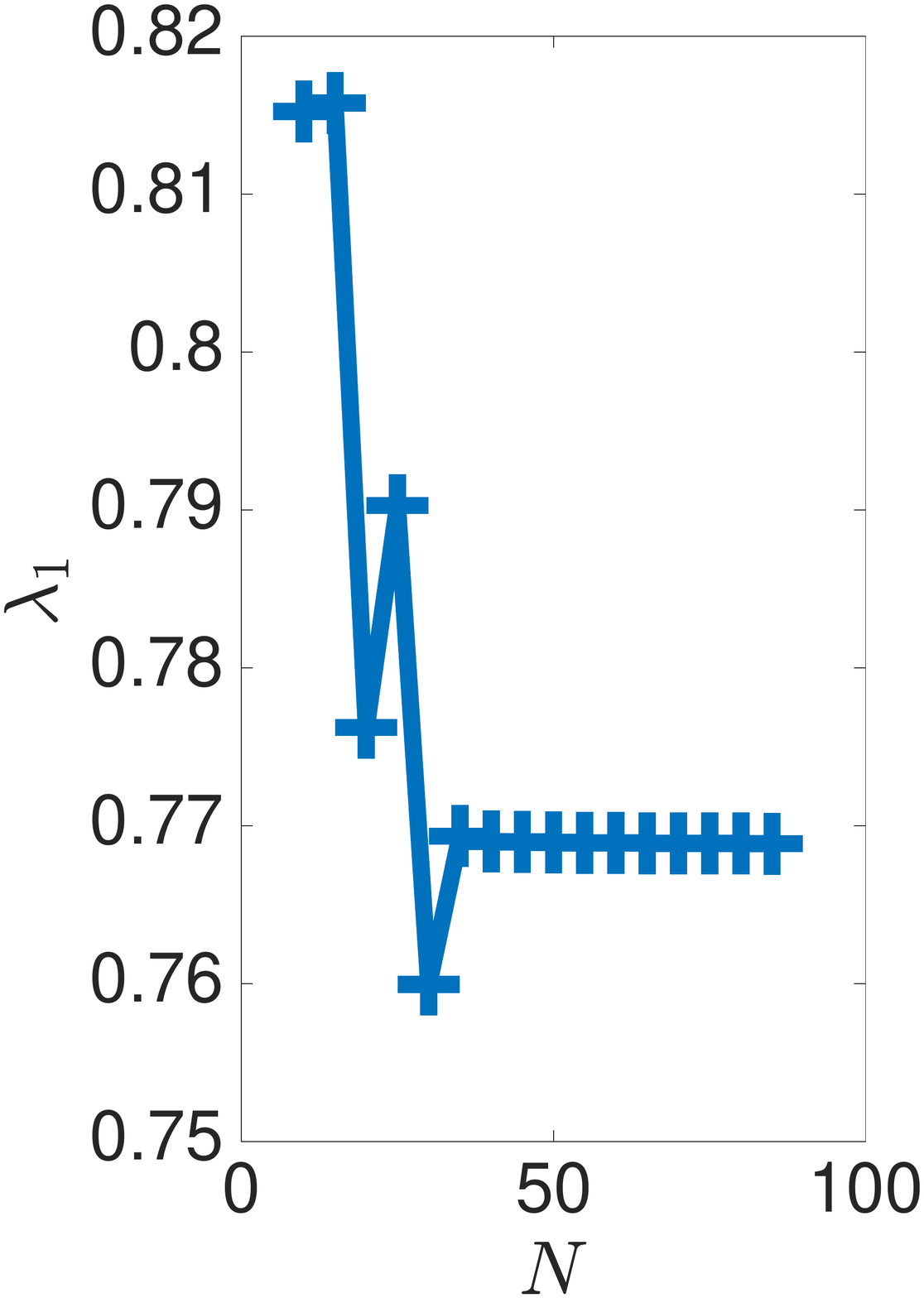}
\includegraphics[width=0.3\textwidth]{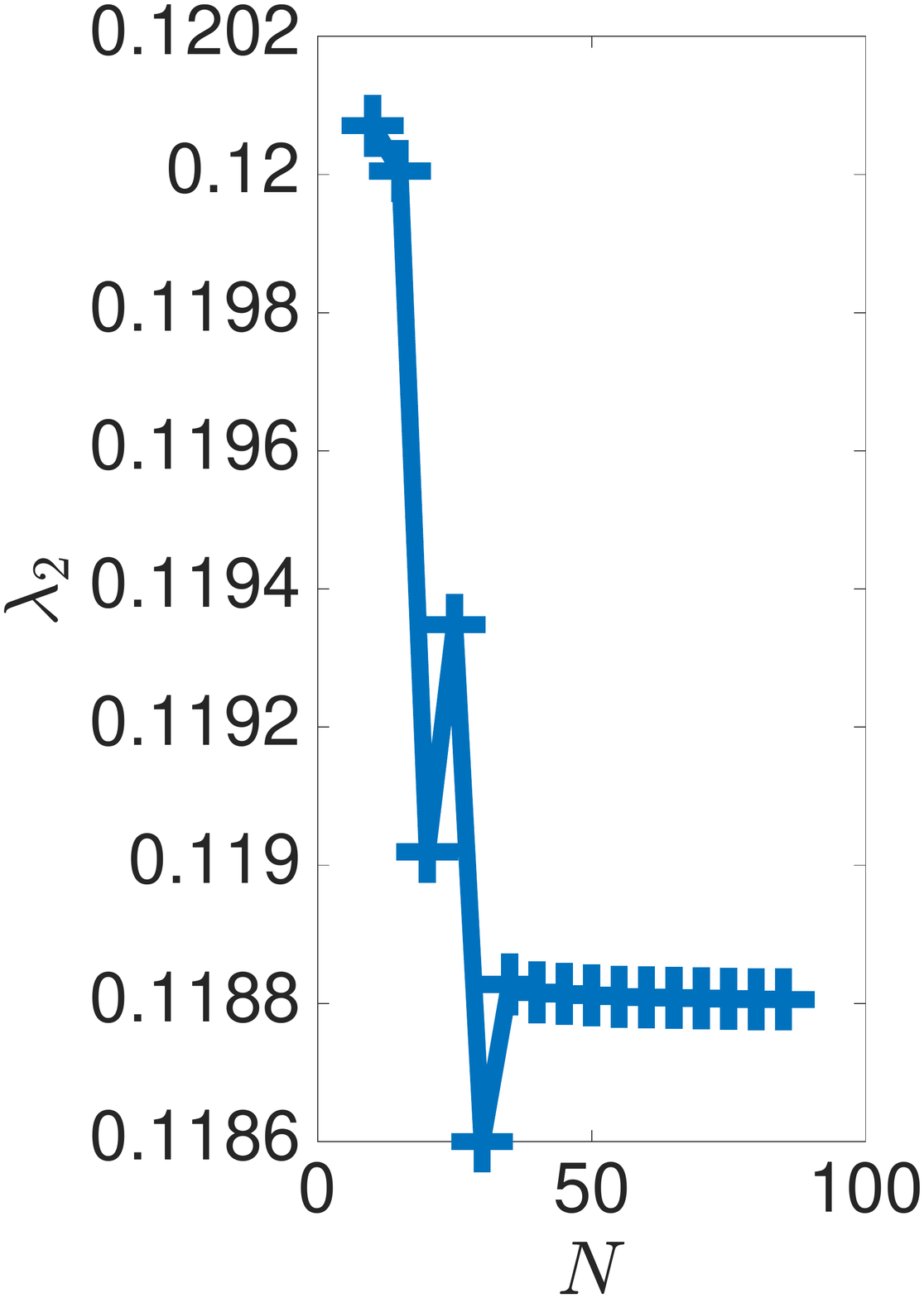}
\includegraphics[width=0.3\textwidth]{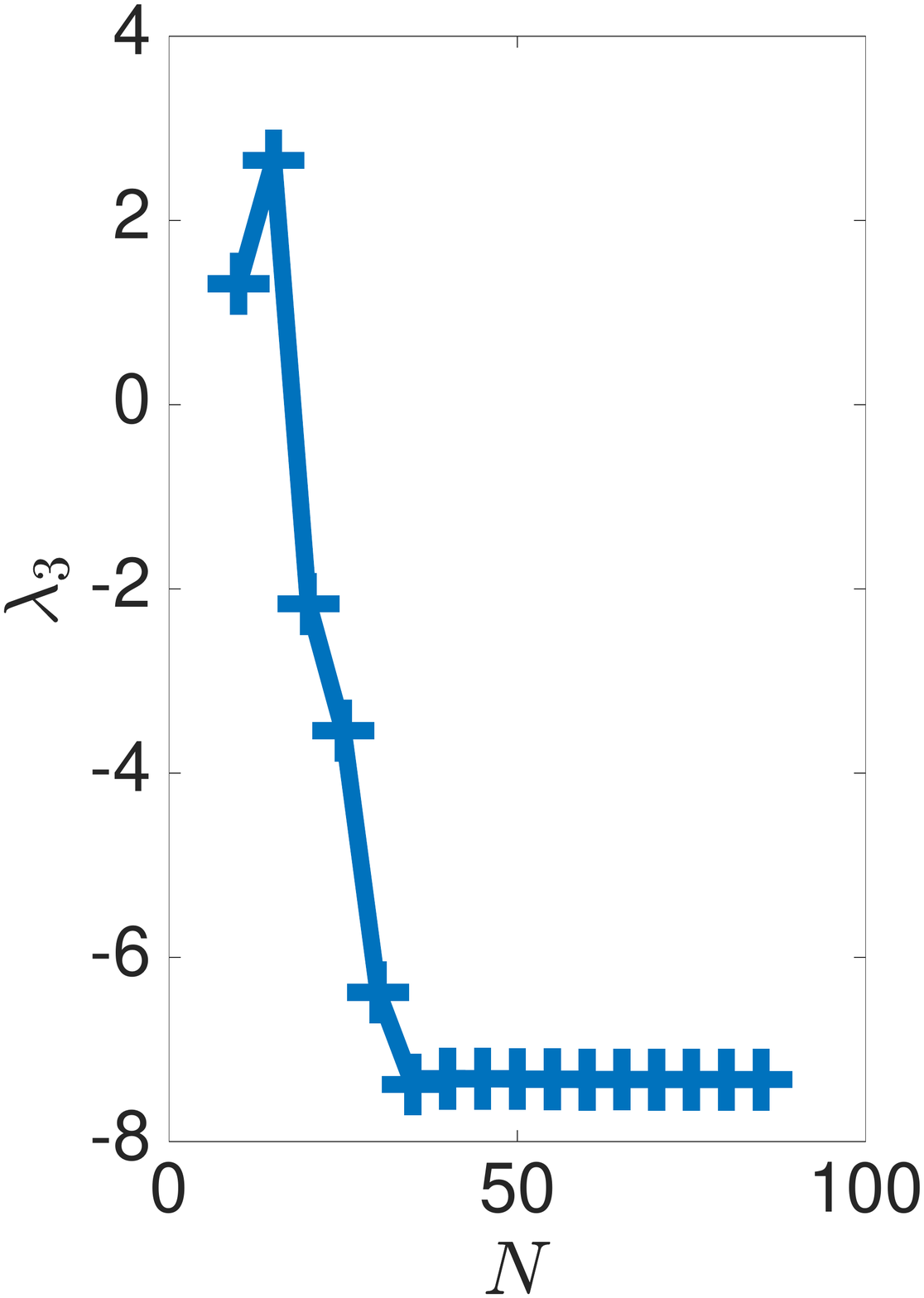}

\caption{Position of the parameters $\lambda_1$, $\lambda_2$ and $\lambda_3$ as a function of the number of inner grid points per dimension $N$ with $N=10,15,\ldots,80,85$. The position of $\lambda$ is obtained by calculating a root of $F^{(3)}(u,\alpha,\bar v,\lambda)$ using Newton iterations until convergence ($\infty$-norm smaller $10^{-9}$). We see that the values converge as $N$ increases.}\label{fig:SWConvergence}
\end{figure}

\begin{figure}
\begin{center}
\includegraphics[width=0.8\textwidth]{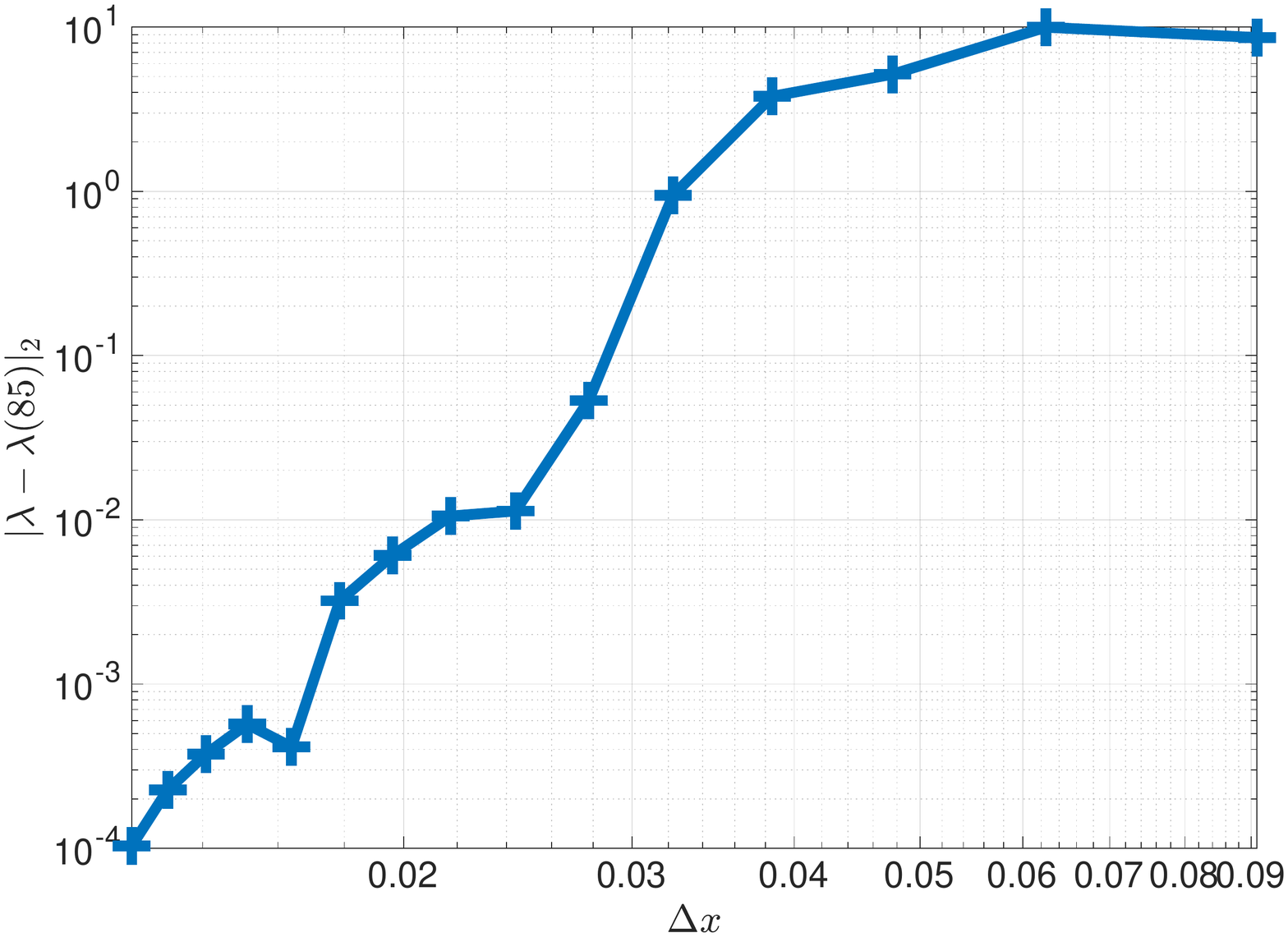}

\end{center}
\caption{Euclidean distance of the parameter $(\lambda_1, \lambda_2, \lambda_3)$ to the calculated value at $N=85$ as a function of the inner grid spacing $\Delta x = \frac 1 {N+1}$. Here $N$ is the number of inner grid points per dimension $N$ with $N=10,15,\ldots,80,85$. In each step the position of $\lambda$ is obtained by calculating a root of $F^{(3)}(u,\alpha,\bar v,\lambda)$ using Newton iterations until convergence ($\infty$-norm smaller $10^{-9}$). We see that the values converge as $\Delta x$ decreases.}\label{fig:SWConvergenceToL85}
\end{figure}

\begin{figure}

\includegraphics[width=0.32\textwidth]{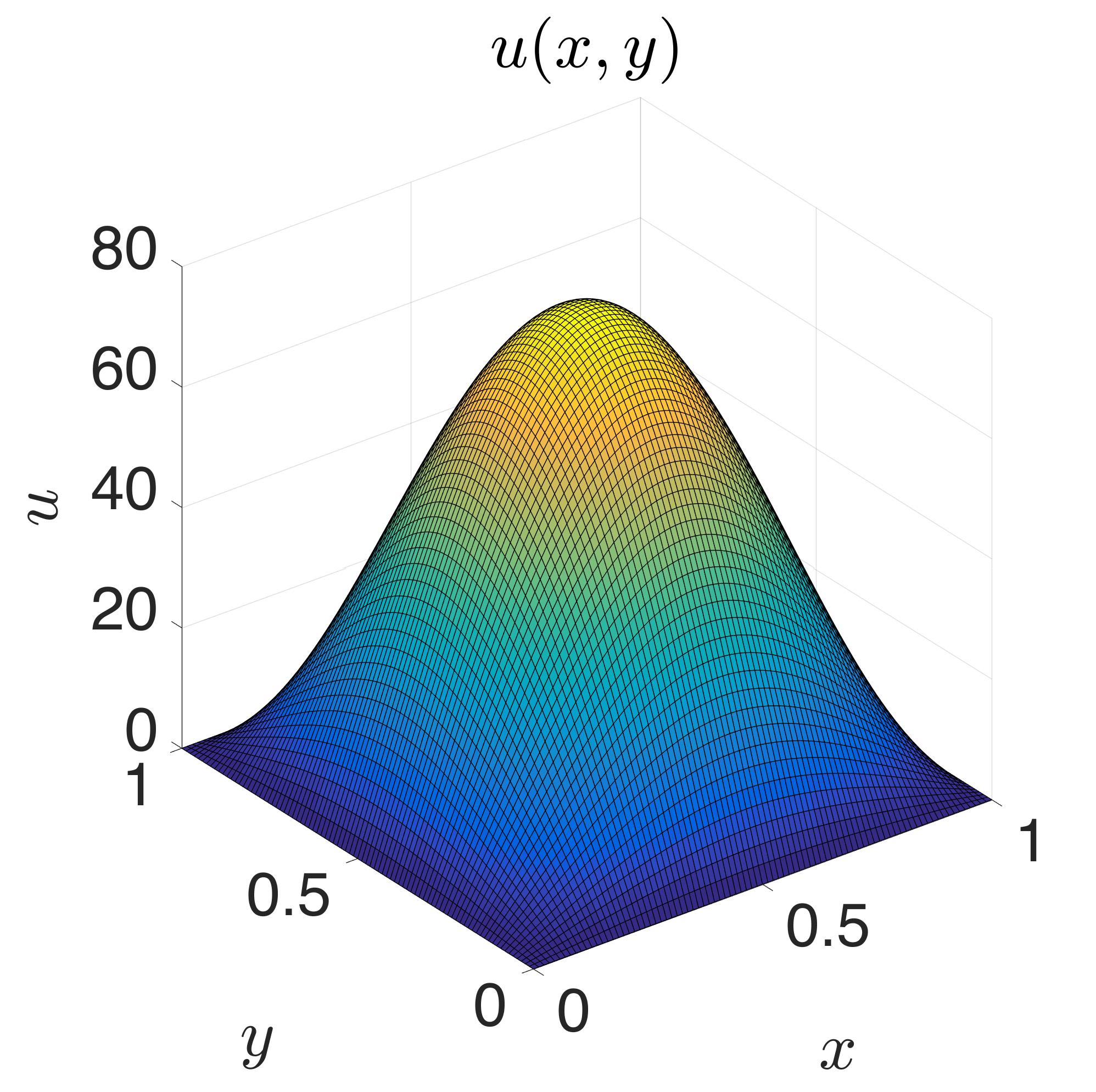}
\includegraphics[width=0.32\textwidth]{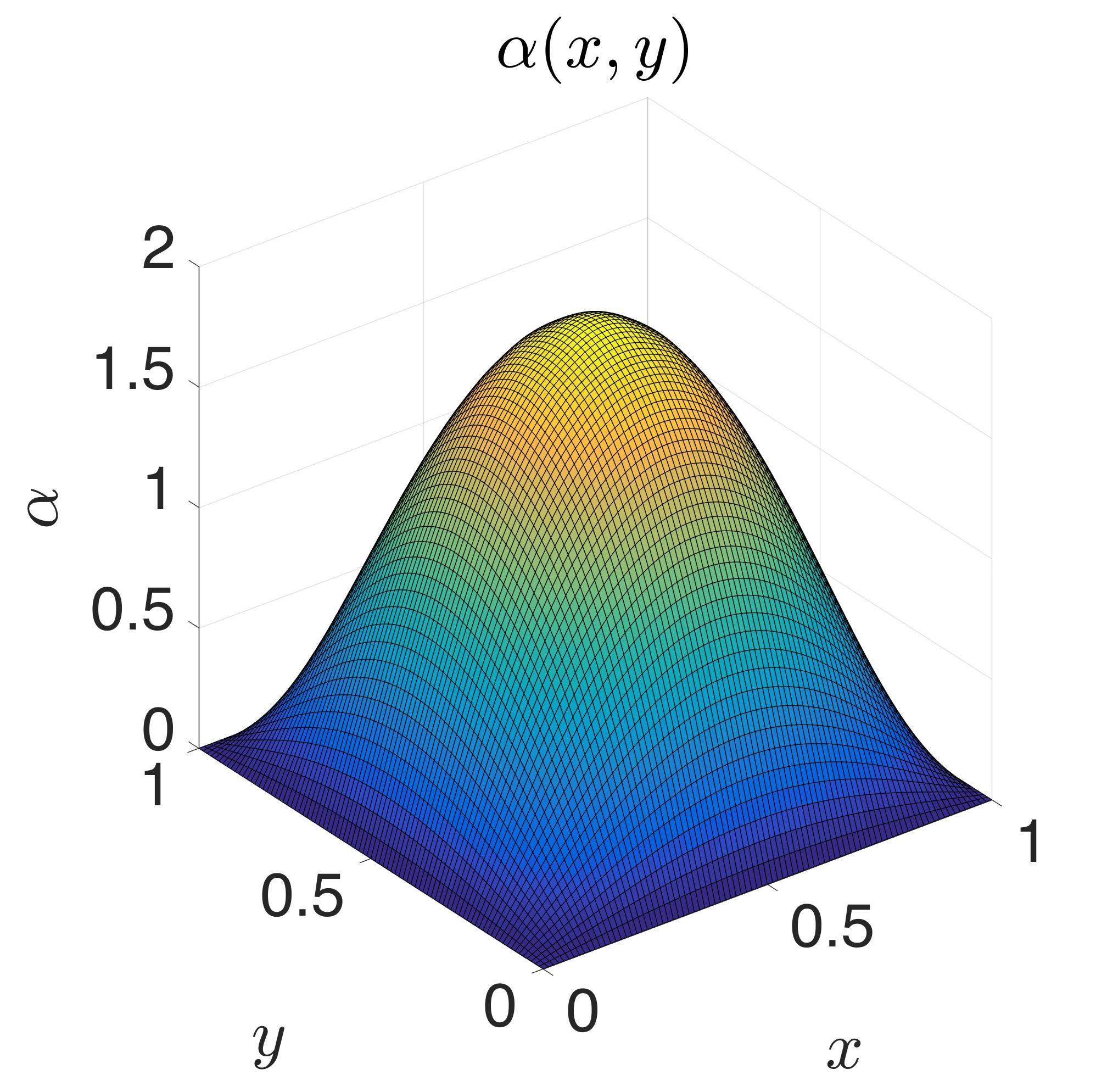}
\includegraphics[width=0.32\textwidth]{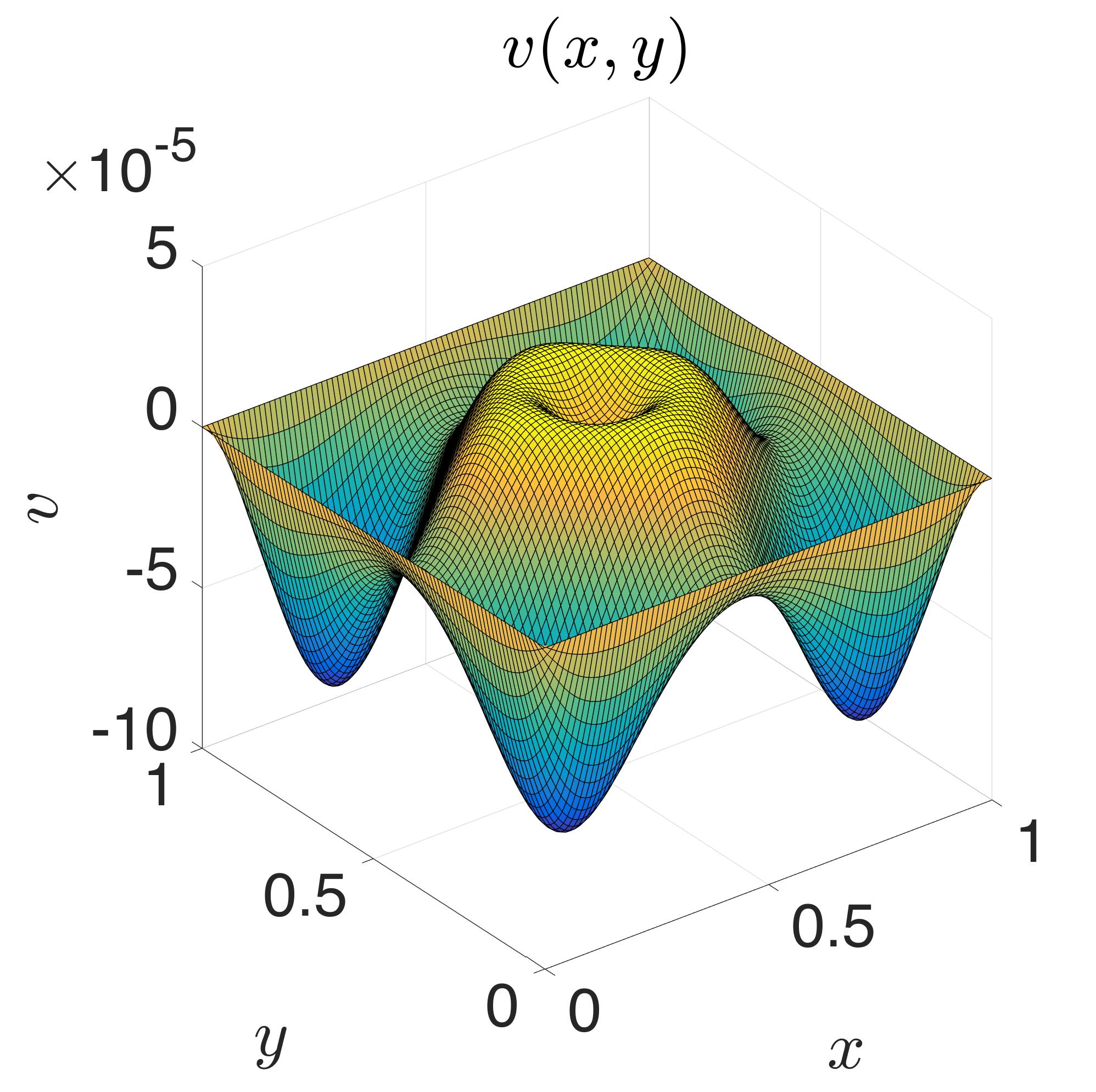}

\caption{Plot of the data $u(x,y)$, $\alpha(x,y)$ and $v(x,y)$ at the approximated swallowtail point for $(N,M)=(85,85)$.}\label{fig:SWSolution}
\end{figure}

\section{Conclusion}\label{sec:Conclusion}

In conclusion, we derived bifurcation test equations for $A$-series singularities of nonlinear functionals and, based on these equations, we developed numerical methods for the  detection of high codimensional branching bifurcations in parameter-dependent PDEs where a variational integrator is used for the discretisation of the problem. This numerical computation  is illustrated by detecting a swallowtail bifurcation in a Bratu-type problem. As part of future research, numerical experiments for finding high codimensional bifurcations for other PDE types can be performed. Another interesting question is the rigorous proof of the convergence rate to the singularity of the numerical method. Besides, bifurcation test equations for $D$-series singularities can be derived and the bifurcations can be numerically detected for different PDE types by considering either variational or non-variational integrators.

\ack
We would like to thank Wolf-Jürgen Beyn and Tom ter Elst for useful feedback on the manuscript and Hinke Osinga, Bernd Krauskopf, and Chris Budd for fruitful discussions.
This research was supported by the Marsden Fund of the Royal Society Te Ap\={a}rangi. 
The authors would like to thank the Isaac Newton Institute for Mathematical Sciences, Cambridge, for support and hospitality during the programme Geometry, compatibility and structure preservation in computational differential equations (2019) where work on this paper was undertaken. This work was supported by EPSRC grant No. EP/K032208/1.
LMK was supported by the UK Engineering and Physical Sciences Research Council (EPSRC) grant EP/L016516/1, the RISE projects CHiPS and NoMADS, and the German National Academic Foundation (Studienstiftung des Deutschen Volkes).
 
\section*{ORCID iDs}
L M Kreusser \url{https://orcid.org/0000-0002-1131-1125}\\
R I McLachlan \url{https://orcid.org/0000-0003-0392-4957}\\
C Offen \url{https://orcid.org/0000-0002-5940-8057}

\section*{References}
\bibliography{resources}{}
\bibliographystyle{iopart-num} 


\end{document}